\documentclass[12pt,reqno]{article}
mm \textheight=9.1in

\usepackage{amssymb}
\usepackage{amsmath}
\usepackage{amsthm}
\usepackage{color}

\newcommand{\br}[3]{{$#1$}$\lower4pt\hbox{$\tp\atop\raise4pt \hbox{$\scriptscriptstyle{#2}$}$} ${$#3$}}
\newcommand{\tw}[3]{{$#1$}${\,\scriptscriptstyle {#2}}\atop\raise9pt\hbox{$\scriptstyle\tp$} ${$#3$}}
\newcommand{\ttps}[2]{{#1}\raise5pt\hbox{$\lower12pt\hbox{$\scriptstyle\tp$}\atop \lower0pt\hbox{$\tilde\;$}$}\raise4.5pt\hbox{${\scriptstyle{#2}}$}}
\newcommand{\st}[1]{\mbox{${\,\scriptscriptstyle {#1}}\atop\raise5.5pt\hbox{$*$}$}}

\newcommand{\rd}[1]{\mbox{${\,\scriptscriptstyle {#1}}\atop\raise5.5pt\hbox{$\bullet$}$}}
\newcommand{\rt}[1]{\otimes_\chi}
\newcommand{\lt}[1]{\mbox{${\,\scriptscriptstyle {#1}}\atop\raise5.5pt\hbox{$\ltimes$}$}}
\newcommand{\btr}{\raise1.2pt\hbox{$\scriptstyle\blacktriangleright$}\hspace{2pt}}
\newcommand{\btl}{\raise1.2pt\hbox{$\scriptstyle\blacktriangleleft$}\hspace{2pt}}

\newcommand{\lcr}{\raise1.0pt \hbox{${\scriptstyle\rightharpoonup}$}}
\newcommand{\rcr}{\raise1.0pt \hbox{${\scriptstyle\leftharpoonup}$}}

\newcommand{\ttp}{{\lower12pt\hbox{$\tp$}\atop \hbox{$\tilde\;$}}}

\newcommand{\id}{\mathrm{id}}

\newcommand{\Tc}{\mathcal{T}}

\newcommand{\Fin}{\mathrm{Fin}}
\newcommand{\Proj}{\mathrm{Proj}}
\newcommand{\Krm}{\mathrm{K}}

\newcommand{\Bc}{\mathcal{B}}

\newcommand{\Ac}{{\mathcal{A}}}

\newcommand{\A}{{A}}

\newcommand{\Ru}{\mathcal{R}}

\newcommand{\Fc}{\mathcal{F}}

\newcommand{\Q}{\mathcal{Q}}
\renewcommand{\O}{\mathcal{O}}

\newcommand{\C}{\mathbb{C}}
\newcommand{\Sbb}{\mathbb{S}}
\newcommand{\Qbb}{\mathbb{Q}}

\newcommand{\Z}{\mathbb{Z}}

\newcommand{\N}{\mathbb{N}}

\newcommand{\tp}{\otimes}

\newcommand{\zt}{\zeta}

\newcommand{\U}{U}

\newcommand{\ve}{\varepsilon}
\newcommand{\gm}{\gamma}

\newcommand{\dt}{\delta}

\newcommand{\op}{\oplus}
\newcommand{\la}{\lambda}
\newcommand{\tr}{\triangleright}
\newcommand{\tl}{\triangleleft}

\newcommand{\Char}{\mathrm{ch }}

\newcommand{\End}{\mathrm{End}}

\newcommand{\Span}{\mathrm{Span}}

\newcommand{\Hom}{\mathrm{Hom}}

\newcommand{\Ind}{\mathrm{Ind}}

\newcommand{\Rm}{\mathrm{R}}

\newcommand{\ad}{\mathrm{ad}}

\newcommand{\La}{\Lambda}

\newcommand{\g}{\mathfrak{g}}

\renewcommand{\k}{\mathfrak{k}}
\newcommand{\h}{\mathfrak{h}}

\newcommand{\s}{\mathfrak{s}}

\renewcommand{\o}{\mathfrak{o}}
\newcommand{\n}{\mathfrak{n}}

\newcommand{\eps}{\epsilon}

\newcommand{\nn}{\nonumber}
\newcommand{\p}{\mathfrak{p}}
\renewcommand{\l}{\mathfrak{l}}

\newcommand{\si}{\sigma}
\newcommand{\al}{\alpha}
\newcommand{\bt}{\beta}

\newcommand{\be}{\begin{eqnarray}}
\newcommand{\ee}{\end{eqnarray}}

\newtheorem{thm}{Theorem}[section]
\newtheorem{propn}[thm]{Proposition}
\newtheorem{lemma}[thm]{Lemma}
\newtheorem{corollary}[thm]{Corollary}
\newtheorem{conjecture}{Conjecture}

\newtheorem{remark}[thm]{Remark}

\newcount\prg

\newcommand{\parag}{\advance\prg by1 {\noindent\bf\thesection.\the\prg\hspace{6pt}}}

\begin{document}

\title{Equivariant vector bundles over even quantum spheres}
\author{
A. Mudrov\footnote{This study is supported  by the RFBR grant 15-01-03148.}
\vspace{20pt}\\
{\em \small  In memory of Ludwig Faddeev}
\vspace{10pt}\\
\small Department of Mathematics,\\ \small University of Leicester, \\
\small University Road,
LE1 7RH Leicester, UK\\
\small e-mail: am405@le.ac.uk
\\
}

\date{ }

\maketitle
\begin{abstract}
We quantize $SO(2n+1)$-equivariant vector bundles on an even complex sphere $\mathbb{S}^{2n}$
as one-sided projective modules over its quantized coordinate ring.
We realize them in two different  ways: as  linear maps between pseudo-parabolic
 modules and as induced modules of the orthogonal quantum group.
Based on this alternative, we study representations of a quantum symmetric pair related to $\mathbb{S}^{2n}_q$
and prove their complete reducibility.
\end{abstract}
\\
{\small \underline{AMS classification codes}: 17B10, 17B37, 53D55.}
\newpage
\section{Introduction}
In this paper we construct one-sided projective  modules over the quantized coordinate ring of the
even sphere $\Sbb^{2n}$, $n\geqslant 1$,
that are equivariant with respect to
the quantized orthogonal group $SO(2n+1)$. Such modules can be viewed as quantized equivariant vector bundles
naturally extending the context of the conventional quantization programme for Poisson manifolds \cite{BFFLS}.

Spheres are important manifolds in the classical geometry and topology \cite{H}, as well as for applications in other areas
\cite{Schw}.
There has been a persistent interest to their non-commmutative counterparts \cite{P,VS,FRT,Dab,DL,LM} since the invention of non-commutative geometry
\cite{C} and quantum groups \cite{D}.  Although in some aspects all quantum spheres admit a uniform description,
e.g. as subvarieties of quantum Euclidian planes \cite{FRT}, the even-dimensional case has been less covered
in the literature. The main reason is   a general effect when a pair
(group,subgroup) determining a classical homogeneous space may be destroyed in quantization. This is observed
for spheres as homogeneous spaces of orthogonal groups, however the odd dimensional case allows for reduced symmetries
that survive quantization. Along these lines, $\Sbb^{2n-1}_q$ were approached via the quantum unitary groups in \cite{VS}
and $\Sbb^{4n-1}_q$ via the quantum symplectic group in \cite{LPR}. Projective modules over odd quantum spheres
were recently studied in \cite{Sheu}. On the contrary, $\Sbb^{2n}$
lacks its quantum stabilizer subgroup but benefits from being a conjugacy class of a connected group,
$G=SO(2n+1)$. This enables one with the machinery of the  category $\O$
to construct and study $\Sbb^{2n}_q$ via its faithful representation on a highest weight module, \cite{M2}.
In the present paper we extend that study to equivariant vector bundles over  $\Sbb^{2n}$.

Semi-simple conjugacy classes of a simple Lie group fall into two classes depending on the type of their stabilizer
subgroup.
A Levi subalgebra $\k$ of a simple Lie algebra $\g$ has a basis of simple roots that is a part of the total basis.
As a consequence, its universal enveloping algebra $U(\k)$  is quantized as a Hopf subalgebra in the total quantum group
$U_q(\g)$.
A regular approach to equivariant star product quantization of vector bundles over
conjugacy classes  with such stabilizer subgroups was developed in \cite{DM}.
However, if some simple roots of $\k$ are not simple with respect to the fixed polarization of $\g$, it cannot be quantized as a subgroup, which disadvantage
makes this case rather challenging.
Nevertheless it is known that  the trivial vector bundles (algebras of functions) on such classes are quantizable,  \cite{M2}.
This fact suggests that a uniform quantization scheme should
exist for all vector bundles on semi-simple classes of  all types (this conjecture is also supported by a local   star product
on $\Sbb^{2n}$ constructed in \cite{M1}, similarly to the Levi case \cite{EEM}). In the present paper we
implement this programme for all equivariant vector bundles on $\Sbb^{2n}$.

The Poisson structure on $\Sbb^{2n}$ under study is restricted from the Semenov-Tian-Shansky bracket on $G$, which is related to the Reflection Equation \cite{KS}.
It makes $G$ a Poisson-Lie manifold over the  Poisson group $G$ endowed with the Drinfeld-Sklyanin bracket, with respect to the conjugation action.
The quantum sphere is treated as a subvariety of $G_q$ and assumes an (conjugation) action of the corresponding quantum group $U_q(\g)$, $\g=\s\o(2n+1)$, on the quantized
 polynomial ring $\C_q[\Sbb^{2n}]$ (we mean the Poisson group $G$ that gives rise to standard  $U_q(\g)$, \cite{ChP}).

One possible realization  of $\C_q[\Sbb^{2n}]$ represents it by
linear operators on  a special  highest weight $U_q(\g)$-module $M$,  which we call base module \cite{M2}.
To  construct an equivariant projective $\C_q[\Sbb^{2n}]$-module, one
should find an invariant idempotent $\hat P$ in the algebra $\End(V)\tp \C_q[\Sbb^{2n}]$ for some finite-dimensional $U_q(\g)$-module $V$.
Such an idempotent projects $V\tp M$ onto a direct summand  submodule and therefore can be obtained
from an irreducible decomposition of $V\tp M$, privided certain technical conditions are fulfilled.
In the case of Levi $\k$, submodules of highest weight in $V\tp M$ are parabolically induced from irreducible $U_q(\k)$-submodules in $V$, which in the classical limit
turn to the fibers over the initial point.
Inspite of no $U_q(\k)\subset U_q(\g)$ for non-Levi $\k$, one still may expect
that the classical irreducible $\k$-submodules in $V$ parameterize $U_q(\g)$-submodules of highest weight in $V\tp M$ as in the Levi case.

Thus the problem is to decompose  $V\tp M$ into a direct sum of highest weight modules, and
describe the irreducible summands. The first question is discussed in general  in \cite{M3}, where
a criterion for complete reducibility of $V\tp M$ is formulated in terms of a canonical contravariant form on  $V\tp M$.
This form is the product of unique, up to a scalar,  contravariant forms on $V$ and $M$. Then $V\tp M$ is completely
reducible if and only if the restriction of the canonical form to the span of singular vectors in $V\tp M$ is non-degenerate.
We apply this criterion to the base module $M$ and all finite-dimensional modules $V$
and prove complete reducibility of $V\tp  M$ for all $q$ not a root of unity.
To answer the second question, we define pseudo-parabolic  modules that generalize parabolic modules for Levi subalgebras.
We prove that all highest weight submodules in $V\tp M$ are pseudo-parabolic, at generic $q$.

As  spheres are symmetric spaces, we take advantage of an alternative presentation of $\C_q[\Sbb^{2n}]$ via quantum symmetric pairs. That part of the paper is analogous to \cite{M4} devoted to quantum projective spaces.
There is a 1-parameter family of solutions of the Reflection Equation \cite{KS} associated with $U_q(\g)$.
Every solution defines a
 one-dimensional representation of $\C_q[\Sbb^{2n}]$ (a ``point" in $\Sbb^{2n}_q$) and facilitates its realization as
a subalgebra in the Hopf dual $\Tc$ to  $U_q(\g)$. At the same time, it defines a left coideal subalgebra $\Bc\subset U_q(\g)$
such that $\C_q[\Sbb^{2n}]$ is identified with the subalgebra of $\Bc$-invariants in $\Tc$ under the action  by right translations.
The algebra $\Bc$ is a deformation of $U(\k')$, where $\k'\simeq \k$ is the isotropy Lie algebra  of the  classical limit of the
``quantum
point".

We prove that every finite-dimensional $U_q(\g)$-module $V$ is completely reducible over $\Bc$ and that simple
$\Bc$-submodules are deformations
of irreducible $\k'$-submodules. They are constructed via invariant projectors from $V\tp M$ onto the corresponding
pseudo-parabolic Verma submodules.
Finally, we identify projective left $\C_q[\Sbb^{2n}]$-modules with $\Bc$-invariants in $\Tc\tp X$, where
$X$ is a right $\Bc$-module and $\Tc$ is equipped with the  $\Bc$-action by right translations. It carries a $\U_q(\g)$-action
via the left translations on $\Tc$. This is a deformation of the classical construction of the equivariant vector bundle
associated with the fiber $X$.

The paper is organized as follows.
Section 2 is devoted to certain properties of extremal projectors that we need for construction of singular vectors.
 We use these results for description of irreducible submodules of $V\tp M$ in
Section 3, where we also prove its complete reducibility. Therein we describe the pseudo-parabolic subcategory
associated with the quantum sphere. Section 4 is devoted to equvariant quantum vector bundles.
We present them as projective modules over $\C_q[\Sbb^{2n}]$ using an equivalence with the pseudo-parabolic category
and compute their equivariant  Grothendieck group. Finally, we give a description of vector bundles via invariants
of the coideal subalgebra from the quantum symmetric pair. Appendix A proves that parabolic modules
are locally finite over the underlying quantum Levi subgroups. Appendix B contains some useful facts about
tensor products of lowest and highest weight vectors.

\section{Singular vectors of $U_q\bigl(\s\o(5)\bigr)$}
This is a technical section devoted to certain aspects of representation theory
of the quantum group $U_q\bigl(\s\o(5)\bigr)$.
It is preparatory for the subsequent section, where we define and study
the main tool of our approach to quantum vector bundles: generalized parabolic Verma modules.

For a detailed  exposition  of general quantum groups
the reader is referred to \cite{D,ChP}. A specific reminder of the odd orthogonal quantum group of rank $n>0$ will be given later
in Section \ref{Sec_OrthQG}. Here we start with the simplest interesting case of  $U_q\bigl(\s\o(5)\bigr)$ and study singular vectors  in its modules with the help of extremal projectors shifted by a ``spectral parameter"
\cite{Zh,KT}. We present them in a form that is convenient for
our applications and establish some useful facts that we have not found in the literature.
\subsection{Shifted extremal projector}
In what follows, we use the shortcuts $\bar q=q^{-1}$, $[z]_q=\frac{q^z-q^{-z}}{q-q^{-1}}$, and $[x,y]_a=xy-ayx$ for any $a\in \C$.
The complex parameter $q\not =0$ is assumed to be not a root of unity.

Let $e,f,q^{\pm h}$ denote generators of $U_q\bigl(\s\l(2)\bigr)$ satisfying
$$
q^{\pm h}e=q^{\pm 2}eq^{\pm h}, \quad
q^{\pm h}f=q^{\mp 2}fq^{\pm h}, \quad
[e,f]=[h]_q.
$$
To accommodate certain operators of interest, we need to extend $U_q\bigl(\s\l(2)\bigr)$ to an algebra $\hat U_q\bigl(\s\l(2)\bigr)$. First of all, we extend the Cartan subalgebra to the field  $\C(q^{h})$ of rational functions in $q^h$. Secondly,
we include formal series in $f^ke^m$ of the same weight with coefficients from $\C(q^{h})$.  Such an extension can be done for an arbitrary quantum group, see  \cite{KT} for details.

Define a one-parameter family  $\pi(s)\in \hat U_q\bigl(\s\l(2)\bigr)$ by
\be
\label{translationed_proj}
\pi(s)=\sum_{k=0}^\infty  f^k e^k \frac{(-1)^{k}q^{k(s-h-1)}}{[k]_q!\prod_{i=1}^{k}[s+i]_q}, \quad s\in \C.
\ee
More exactly, it belongs to the  zero weight subalgebra  $\hat U_q^0\bigl(\s\l(2)\bigr)\subset \hat U_q\bigl(\s\l(2)\bigr)$.
The element $\pi(s)$ features the following properties, which can be checked by a direct calculation:
\be
  e\pi(s)=\pi(s+1)\frac{q^{-h}[s+1-h]_q}{[s+1]_q}e,
 \quad
  f\pi(s)=\pi(s-1)\frac{q^{h+2}[s]_q}{[s-2-h]_q}f.
  \label{alpha-intertwiner}
\ee
It follows that $\pi(h+1)$ is the extremal projector, i.e. an idempotent satisfying
$e\pi(h+1)=0$, $\pi(h+1)f=0$, cf. \cite{KT}.
Another corollary is that
 $\pi(s)$ is well defined on every $U_q\bigl(\s\l(2)\bigr)$-module of highest weight,
where it returns
\be
q^{-m\mu(h)+m(m-3)}\prod_{i=1}^{m} \frac{[-s+\mu(h)+i]_q}{[-s-i]_q},
\label{Zhb_fact}
\ee
on a vector $v$ of weight $\mu$ such that $e^{m}v$ is the highest vector.

We consider finite-dimensional $U_q\bigl(\s\l(2)\bigr)$-modules whose weights belong to $q^\Z$. Then the following fact is important for our exposition.
\begin{propn}
  For all $s\in \C$ such that $q^{2s}\in -q^{\Qbb}$, the operator $\pi(s)$ is invertible on every finite-dimensional module.
  \label{pi-invertible}
\end{propn}
\begin{proof}
 Immediate corollary of (\ref{Zhb_fact}) since $\mu(h)\in \Z$ and $q$ is not a root of unity.
\end{proof}

Let $\si$ denote an algebra automorphism of $U_q\bigl(\s\l(2)\bigr)$
acting by the assignment $f\mapsto e$, $e\mapsto f$, $h\mapsto -h$.
Observe that the operator  $\si\bigl(\pi(s)\bigr)$ is well defined on every
finite-dimensional module. We are going to relate it to $\pi(s)$.

To that end, we extend  $\s\l(2)$ to $\g=\s\l(3)$
and use a fact that $\pi(s)$ is essentially a unique element satisfying a certain identity in  $\hat U_q(\g)$.
We assume that our $\s\l(2)$ corresponds  to a simple root $\al$. Let $\bt$ be the other simple root and put $\gm=\al+\bt$.

In the algebra $\hat U_q(\g)$, define
$$
f_{\gm}=f_\bt f-q ff_\bt,
\quad
e_\gm=ee_\bt-q^{-1}e_\bt e,
$$
$$
\hat f_{\gm}(s)=f_\bt f -ff_\bt \frac{[s]_q}{[s+1]_q},
\quad
\hat e_\gm(s)=ee_\bt-e_\bt e\frac{[s]_q}{[s+1]_q}.
$$
The vectors $e_\gm$ and $f_\gm$ form a quantum $\s\l(2)$-pair in $U_q(\g)$.
\begin{propn}
\label{factor}
The element $\pi(s)$ satisfies the equality
$$\pi(s) f_{\gm}=\hat f_{\gm}(s)\pi(s+1).$$
Furthermore, $\pi(h-s)$ is  a unique element from $\hat U_q^0\bigl(\s\l(2)\bigr)$, up to a factor  $a(h-s)\in \hat U_q(\h)$,
satisfying
\be
\hat e_\gm(s) \pi(h-s)= \pi(h-s-1)e_\gm.
\ee
 \end{propn}
\begin{proof}
Direct calculation.
\end{proof}
Define a family $C(m,s)$ of rational trigonometric functions of $s$ parameterized by $m\in \Z_+$ (the set of non-negative integers):
$$
C(m,s)=\sum_{k=0}^{\infty}(-1)^{k}q^{k(s+m -1)}\prod_{i=1}^{k}\frac{[m-i+1]_q}{[s+i]_q}.
$$
Note that all terms with $k>m$ vanish, so that the sum is finite.
\begin{lemma}
For all $m\in \Z_+$, $C(m,s)=q^{-m}\frac{[s]_q}{[s+m]_q}$.
\end{lemma}
\begin{proof}
Obviously the statement is true for $m=0$, as the summation over $k>0$ turns zero.
It is easy to check that the right-hand side satisfies the identity $C(m+1,s)=1-\frac{q^{s+m}[m+1]_q}{[s+1]_q}C(m,s+1)$.
Then induction on $m$ employing this equality proves the formula for $m>0$.
\end{proof}

\begin{propn}
\label{lowest-to-highest}
  The operator identity $\si\bigl(\pi(s)\bigr)=q^{-h}\frac{[s]_q}{[s+h]_q}\pi(h+s)$
  holds true in every finite-dimensional representation of $U_q\bigl(\s\l(2)\bigr)$.
\end{propn}
\begin{proof}
Put $\tilde \pi(s)=\si\bigl(\pi(-s)\bigr)$.
Applying $\si$ to the first equality in Proposition \ref{factor} and changing the sign of the parameter $s$ we get
\be
\hat e_\gm(s)\frac{1}{[s]_q}q^{h}\tilde \pi(s)=\frac{1}{[s+1]_q}q^{h}\tilde \pi(s+1) e_\gm
\quad\Rightarrow \quad \frac{1}{[s]_q}q^{h}\tilde \pi(s)=a(h-s) \pi(h-s),
\label{coefficient}
\ee
due to uniqueness,
cf. Proposition \ref{factor}. To complete the proof, we need to calculate the factor $a(h-s)$. We
do it by evaluating  (\ref{coefficient}) on a subspace of weight $\mu(h)=m\in \Z$.

Since $\tilde \pi(s)=C(h,-s)\mod \hat U_q\bigl(\s\l(2)\bigr)e$ on every non-negative weight, we have
$$
a(m-s)=-q^{m}\frac{1}{[s]_q}\frac{q^{-m}[s]_q}{[m-s]_q}=\frac{1}{[s-m]_q}, \quad m\geqslant 0.
$$
Applying $\si$ to $\tilde \pi(s)$ and evaluating the result
on $m<0$,  one arrives at
$\pi(s)=C(-h,s)\mod \hat U_q\bigl(\s\l(2)\bigr)f$
and hence at
$\pi(h-s)=C(-h,h-s)\mod \hat U_q\bigl(\s\l(2)\bigr)f$. So, for negative $m$,
the equality (\ref{coefficient}) turns to
$\frac{1}{[s]_q}q^{m}=a(m-s)C(-m,m-s)$ producing
the same result.
Therefore $a(h-s)=\frac{1}{[s-h]_q}$ on all weight vectors.
\end{proof}

\subsection{Dynamical root vectors of $U_q(\s\o(5)$}
To the end of the section we assume $\g=\s\o(5)$.
We express its simple roots $\al=\ve_1$ and $\bt=\ve_2-\ve_1$
through the orthogonal basis of short roots $\{\ve_1,\ve_2\}\in \h^*$ normalizing the inner product so that
$(\ve_i,\ve_i)=1$, $i=1,2$. We introduce  ``compound root vectors"
for $\gm=\ve_2$ and $\dt=\ve_1+\ve_2$  by
$$
f_{\gm}=[f_{\al},f_{\bt}]_{\bar q}
,\quad
e_{\gm}=[e_{\bt},e_{\al}]_{q},
$$
$$
f_{\dt}=[[f_{\al},f_{\bt}]_{\bar q},f_{\al}]
,\quad
e_{\dt}=[e_{\al},[e_{\bt},e_{\al}]_{q}].
$$
Along with $q^{\pm h_\gm}$ and $q^{\pm h_\dt}$, these pairs generate subalgebras isomorphic to  $U_q\bigl(\s\l(2)\bigr)$.
We denote these subalgebras by $\g^{(\eta)}$ for each positive root $\eta$.
Let $\g_+$ denote the linear span of $e_\al,e_\bt,e_\gm, e_\dt$, and  define $\g_-$ similarly.

We will also need a ``dynamical" version of the vectors $f_\dt$ and $e_\dt$:
\be
\hat f_\dt(s)&=&\bar p^2\Bigl( f_\al^2 f_\bt \frac{[s]_p}{[s+2]_p}-f_\al f_\bt f_\al[2]_p\frac{[s]_p}{[s+1]_p}+ f_\bt f_\al^2\Bigr),
\\
\hat e_\delta(s)&=&  e_\al^2 e_\bt -e_\al e_\bt e_\al[2]_p\frac{[s+2]_p}{[s+1]_p}+e_\bt e_\al^2\frac{[s+2]_p}{[s]_p},
\ee
where $p=q^{\frac{1}{2}}$ and $\bar p=p^{-1}$.
Introduce a one-parameter family $\pi_\al(s)\in \hat U_q(\g)$ by
$$
\pi_\al(s)=\sum_{k=0}^\infty  f^k_\al e^k_\al \frac{[2]_p^k(-1)^{k}p^{k(s-2h_\al-1)}}{[k]_p!\prod_{i=1}^{k}[s+i]_p}.
$$
It is the image of $\pi(s)$ under the embedding $\s\l(2)\subset \g$ corresponding to the short root $\al$
(the parameter $q$ in the preceding subsections is to be replaced with $q_\al=p$).
\begin{propn}
\label{dyn_f}
The following relations hold true:
\be
\pi_\al(s) f_{\dt}&=& \hat f_{\dt}(s)\pi_\al(s+2),
\label{f-pi-int}
\\
\pi_\al(2h_\al+s)e_{\dt}&=&\hat e_\dt(s)\pi_\al(2h_\al+s+2)q.
\label{e-pi-int}
\ee
\end{propn}
\begin{proof}
Equality (\ref{f-pi-int}) is checked through a straightforward  calculation.
Furthermore, the automorphism $\si$ takes $f_\dt$ to $\bar qe_\dt$
and $\hat f_\dt(s)$ to $\bar q\frac{[s]_p}{[s+2]_p}\hat e_\dt(s)$. Applying it to
(\ref{f-pi-int}) we get
$\si\bigl(\pi_\al(s)\bigr)e_{\dt}= \frac{[s]_p}{[s+2]_p}\hat e_\dt(s)\si\bigl(\pi_\al(s+2)\bigr).$
Substitution of   $\si\bigl(\pi_\al(s)\bigr)=\frac{p^{-2h_\al}[s]_p}{[s+2h_\al]_p}\pi_\al(2h_\al+s)$, thanks to
Proposition \ref{lowest-to-highest}, proves (\ref{e-pi-int}) in all finite-dimensional representations and therefore  in $\hat U_q(\g)$.
\end{proof}

\subsection{On singular vectors in Verma modules}
Recall that a weight vector in a $U_q(\g)$-module is called singular if it is annihilated by generators of $U_q(\g_+)$.
In this section we  derive explicit formulas for singular vectors of weights  $\mu-m\dt$, $m\in \Z_+$, in a Verma module of highest weight $\mu$ subject to certain conditions.

\begin{lemma}
  One has $f_\al\hat f_\dt(s+1)=\hat f_\dt(s)\frac{[s+1]_p}{[s+3]_p}f_\al.$
\label{din_dt-al_com}
\end{lemma}
\begin{proof}
The Serre relations imply $f_\al f_\dt=p^2 f_\dt f_\al$, cf. Section \ref{Sec_OrthQG}.
Applying Proposition \ref{dyn_f} and then formula (\ref{alpha-intertwiner})to $ f_\al\hat f_\dt(s)\pi_\al(s+2)$  we write it as
$$
 f_\al\pi_\al(s) f_{\dt}
=\frac{p^{2h_\al+2}[s]_p}{[s-2-2h_\al]_p}\pi_\al(s-1)p^2 f_{\dt}  f_\al
=\frac{p^{2h_\al+4}[s]_p}{[s-2-2h_\al]_p} \hat f_{\dt}(s-1)\pi_\al(s+1)  f_\al
$$
$$
=\frac{p^{2h_\al+4}[s]_p}{[s-2-2h_\al]_p} \hat f_{\dt}(s-1)\frac{[s-2h_\al]p^{-2h_\al-2}}{[s+2]_p}f_\al\pi_\al(s+2)
=\frac{[s]_p}{[s+2]_p}\hat f_{\dt}(s-1)f_\al\pi_\al(s+2).
$$
Evaluating these on the highest vector of a generic Verma module and replacing $s$ with $s+1$ we prove the lemma.
\end{proof}
As a consequence, we obtain the formula
\be
\hat f_\dt(s)f_\al^2=\frac{[s+4]_p[s+3]_p}{[s+2]_p[s+1]_p}f_\al^2\hat f_\dt(s+2).
\label{aux2}
\ee
We will also need the  commutation relation
\be
[e_\bt ,\hat f_\dt(s)]=
 f_\al^2 \frac{[2]_p[s+h_\bt+2]_q\bar q}{[s+2]_p[s+1]_p},
\label{aux1}
\ee
which readily follows from the defining relations of $U_q(\g)$.

\begin{propn}
Let $\hat M_\mu$ be a Verma module of highest weight $\mu$ with the highest vector $1_\mu$.
Suppose that $\mu$ satisfies the condition
$
[(\mu,\dt)-m+2]_q=0
$
for some $m\in \N$.
Then the vector $\pi_\al(2h_\al+1) f_{\dt}^m1_\mu$ is singular.
\label{singular}
\end{propn}

\begin{proof}
It was mentioned that $\pi_\al(2h_\al+1)$ is the extremal projector $\hat U_q\bigl(\g^{(\al)}\bigr)$,
therefore $e_\al$ kills $\pi_\al(2h_\al+1) f_{\dt}^m1_\mu$. Explicitly, write it as $\pi_\al(s) f_{\dt}^m1_\mu$
with $s=2(\mu,\al)-2m+1$.
Using  (\ref{alpha-intertwiner}) we transform $e_\al\pi_\al(s) f_{\dt}^m$ to
$$
\pi_\al(s+1)\frac{p^{-2h_\al}[s+1-2h_\al]_p}{[s+1]_p}e_\al f_{\dt}^m=\pi_\al(s+1)q^{-2h_\al}e_\al f_{\dt}^m \frac{[s+1-2h_\al+2m-2]_p}{[s+1]_p}.
$$
Therefore $\pi_\al(s) f_{\dt}^m 1_\mu$ is annihilated by  $e_\al$ indeed.
Furthermore, notice that $e_\bt\pi_\al(s+2m)\in \hat U_q(\g)\g_+$ since $\pi_\al(s+2m)$ has zero weight.
Using the factorization
\be
\pi_\al(s) f_{\dt}^m= q^{-m}\hat f_{\dt}(s)\hat f_{\dt}(s+2)\ldots \hat f_{\dt}(s+2m-2)\pi_\al(s+2m)
\label{factorize f_dt}
\ee
that follows from (\ref{f-pi-int})
we present $e_\bt \pi_\al(s) f_{\dt}^m $ as
$$
\sum_{k=1}^{m} q^{-k}\prod_{i=1}^{k-1}\hat f_{\dt}(s+2i-2)\> [e_\bt,\hat f_\dt(s+2k-2)]\pi_\al(s+2k)f_{\dt}^{m-k} \mod \hat U_q(\g)\g_+.
$$
  Using (\ref{aux1}) for the commutator  we rewrite this sum as
$$
[2]_p\sum_{k=1}^{m}q^{-k-1} \prod_{i=1}^{k-1}\hat f_{\dt}(s+2i-2)\> f_\al^2 \frac{[s+h_\bt+2k]_q}{[s+2k]_p[s+2k-1]_p}\pi_\al(s+2k)f_\dt^{m-k}.
$$
Pushing $f_\al^2$ to the left with the help of (\ref{aux2}) we get
$$
[2]_pq^{-2}f_\al^2 \sum_{k=1}^{m} \frac{[s+h_\bt+2k]_q}{[s+1]_p[s+2]_p}\prod_{i=1}^{k-1}q^{-1}\hat f_{\dt}(s+i+1)\> \pi_\al(s+2k)f_\dt^{m-k}.
$$
Now pushing $\pi_\al$ back to the left with the help of  (\ref{f-pi-int}) we restore the factor $f_\dt^{m-1}$ on the right and
arrive to
$$
[2]_pq^{-2}f_\al^2 \frac{\pi_\al(s+1)}{[s+1]_p[s+2]_p}\sum_{k=1}^{m} [s+h_\bt+2k]_qf_\dt^{m-1}.
$$
The sum here is equal to $[s+h_\bt+m+1]_q[m]_q$. Pushing $\pi_\al(s)$  to the right  with the use of
 (\ref{factorize f_dt}) and applying to the highest vector we eventually obtain
$$
f_\al^2\hat f_{\dt}(s+1)\hat f_{\dt}(s+3)\ldots \hat f_{\dt}(s+2m-3)\frac{[2]_p[m]_q[s+h_\bt+m+1]_q}{[s+1]_p[s+2]_p} q^{-m-1}1_\mu.
$$
Therefore, the vector in question is  annihilated by  $e_\bt$ if $[s+(\mu,\bt)+m+1]_q=0$. Substitution of $s=2(\mu,\al)-2m+1$
translates this to the condition on the value of $(\mu,\dt)$.
\end{proof}
In the modules of our interest appearing in the next section, the highest weight $\mu$ satisfies
$q^{2(\mu,\al)}=q^{2(\mu,\al^\vee)}_\al\in -q^\Z$, where $\al^\vee$ is the coroot $\frac{2\al}{(\al,\al)}$. Hence the
right-hand side in (\ref{factorize f_dt}) is regular at such weights by Proposition \ref{pi-invertible}.

Remark that the assumption on $\mu$ of Proposition \ref{singular} can be written as the condition $[(\mu+\rho,\dt)-m]_q=0$, where $\rho$ is the half-sum of positive roots of $\s\o(5)$. This is a pole of the quantum Shapovalov determinant of the module $\hat M_\mu$, \cite{DCK}.
\section{Generalized parabolic modules}
In this section we introduce and study a class of modules over the odd orthogonal quantum group
that play a role of "representations" for quantized vector bundles. Some of their properties are derived
from corresponding properties of parabolic modules, which are established in Section A.
We show that generalized parabolic modules form a nice category that is equivalent to the category of classical finite-dimensional modules over
the even orthogonal group, which is the stabilizer of the initial point on the even sphere. In
the classical geometry, its modules are fibers over the initial point.
They will also parameterize quantized vector bundles in the subsequent section.
\subsection{The odd orthogonal quantum group}
\label{Sec_OrthQG}
From now on $\g$ is the Lie algebra $\s\o(2n+1)$, $n\in \N$, and $\h\subset \g$ its Cartan subalgebra. An $\ad$-invariant form identifies $\g$
with $\g^*$ and $\h$ with $\h^*$. Let $(\> .\>,\>.\>)$ designate its restriction to $\h$.
Denote by  $\Rm$ its root system and fix the subset of positive roots $\Rm^+$ with basis $\Pi^+$.
We normalize the inner product so that the length of  short roots is $1$.
Then $\Rm^+$ contains an orthogonal basis $\{\ve_i\}_{i=1}^n\subset \h^*$
of short roots such that $\al_1=\ve_1, \al_{i}=\ve_{i}-\ve_{i-1}$, $i=2,\ldots, n$, constitute $\Pi^+$.
For all $\la\in \h^*$ we denote by  $h_\la$ the element of $\h$ such that $\mu(h_\la)=(\mu,\la)$ for arbitrary $ \mu\in \h^*$.

By $U_q(\g)$ we understand the standard  orthogonal quantum group \cite{ChP} over the complex field
 with the set of generators $e_\al$, $f_\al$, and invertible $q^{h_\al}$ satisfying
$$
q^{\pm h_\al}e_\bt=q^{\pm (\al,\bt)}e_\bt q^{\pm h_\al},
\quad
[e_\al,f_\bt]=\dt_{\al,\bt}[h_\al]_q,
\quad
q^{\pm h_\al}f_\bt=q^{\mp (\al,\bt)}f_\bt q^{\pm h_\al},\quad \al,\bt \in \Pi^+.
$$
The generators $e_{\al}$ and  $f_\al$, $\al \in \Pi^+$, obey the Serre relations,
$$
[e_{\al},[e_{\al},e_{\bt}]_q]_{\bar q}=0, \quad [f_{\al},[f_{\al},f_{\bt}]_q]_{\bar q}=0, \quad \forall \al,\bt\in \Pi^+\quad\mbox{s.t.}\quad
 \frac{2(\al,\bt)}{(\al,\al)}=-1,
$$
$$
[e_{\al},e_{\bt}]=0,\quad[f_{\al},f_{\bt}]=0,\quad \forall \al,\bt\in \Pi^+\quad\mbox{s.t.}\quad
 (\al,\bt)=0,
$$
$$
[e_{\al_1},\hat e_{\dt}]=0, \quad [f_{\al_1},\hat f_{\dt}]=0,
$$
where $\hat e_{\dt}=[e_{ \al_1},[e_{\al_1},e_{\al_2}]_q]_{\bar q}$ and
 $\hat f_{\dt}=[f_{ \al_1},[f_{\al_1},f_{\al_2}]_q]_{\bar q}$.
Note that $\hat e_{\dt}$ and $\hat f_{\dt}$ are deformations of classical root vectors.
Although they do not form a quantum $\s\l(2)$-pair,  they  play a role in what follows.

Fix the comultiplication on the generators of  $U_q(\g)$ as
$$\Delta(f_\al)= f_\al\tp 1+q^{-h_\al}\tp f_\al,\quad\Delta(q^{\pm h_\al})=q^{\pm h_\al}\tp q^{\pm h_\al},\quad\Delta(e_\al)= e_\al\tp q^{h_\al}+1\tp e_\al.$$
Then the antipode $\gm$ acts by the assignment
$$\gm( f_\al)=- q^{h_\al}f_\al, \quad \gm( e_\al)=- e_\al q^{-h_\al}, \quad \gm(q^{\pm h_\al})=q^{\mp h_\al}.$$
 The counit $\eps$ returns $\eps(e_\al)=0=\eps(f_\al)$, and $\eps(q^{\pm h_\al})=1$.

Denote by $U_q(\h)$,  $U_q(\g_+)$, and  $U_q(\g_-)$  the subalgebras generated by $\{q^{\pm h_\al}\}_{\al\in \Pi^+}$, $\{e_\al\}_{\al\in \Pi^+}$, and $\{f_\al\}_{\al\in \Pi^+}$, respectively.
The Lie subalgebra $\l=\g\l(n)\supset \h$ with the basis of simple roots $\Pi^+_\l=\{\al_i\}_{i=2}^n$ is a Levi subalgebra in $\g$.
Its universal enveloping algebra is quantized as a quantum subgroup $U_q(\l)\subset U_q(\g)$.
On the contrary, the Lie subalgebra $\k=\s\o(2n)\subset \g$ with the basis of simple roots $\Pi^+_\k=\{\dt,\al_2,\ldots,\al_n\}$ does
not have a natural quantum analog of $U(\k)\subset U(\g)$.

For each $\al\in \Rm^+$ we denote by $\g^{(\al)}$ the corresponding  $\s\l(2)$-subalgebra in $\g$. If $\al$ is simple,
then $U_q(\g^{(\al)})$ is a quantum subgroup in $U_q(\g)$. For compound $\al$,
there is a triple of elements (not unique) in $U_q(\g)$ generating a $U_q\bigl(\s\l(2)\bigr)$-subalgebra (but not a Hopf one),
which is a deformation of $U(\g^{(\al)})$.
Its generators of weights $\pm \al$ enter a Poincare-Birkhoff-Witt (PBW) system delivering a basis in $U_q(\g_\pm)$, \cite{ChP}. We denote this subalgebra by $U_q(\g^{(\al)})$
assuming its root vectors fixed.

Given a diagonalizable $U_q(\h)$-module $W$ with weight subspaces $W[\mu]$, denote by $\Char(W)$ the formal sum $\sum_{\mu} \dim W[\mu] q^\mu$, where
$q^\mu$ is a one-dimensional representation of $U_q(\h)$ acting by $q^\mu\colon q^{h_\al}\mapsto q^{(\mu,\al)}$.
We also use this definition for classical $\h$-modules (replacing $e^\mu$ with $q^\mu$) with the obvious definition of weight subspaces.

For two $U_q(\h)$-modules $W_i$, $i=1,2$, we write $\Char(W_1)\leqslant \Char(W_2)$  if $\dim W_1[\mu] \leqslant \dim W_2[\mu]$
for all $\mu\in \h^*$.
The usual rules for characters hold true:
given a short exact sequence $0\to W_1\to W_2\to W_3\to 0$ one has $\Char(W_2)=\Char(W_1)+\Char(W_3)$.
Obviously $\Char(W_1)\leqslant \Char(W_2)$ and $\Char(W_3)\leqslant\Char(W_2)$.
Also, $\Char(W_1\tp W_2)=\Char(W_1)\Char(W_2)$ for any $W_1$ and $W_2$.

Recall that finite-dimensional $U_q(\g)$-modules are  $U_q(\h)$-diagonalizable, and the eigenvalues of $q^{h_\al}$
belong to $\pm q^{\Z}$. Throughout the text, we mean by finite-dimensional $U_q(\g)$-modules only those
whose $q^{h_\al}$-eigenvalues belong to $q^{\Z}$ and denote their category by $\Fin_q(\g)$. Such modules are in one-to-one
correspondence with finite-dimensional $U(\g)$-modules
and have the same characters. This weight convention does not apply to
finite-dimensional $U_q(\l)$-modules and infinite-dimensional $U_q(\g)$-modules.

The algebra $U_q(\g)$ and its modules appearing in this exposition rationally depend on the parameter $q$
which can be specialized to complex numbers except for a certain set of values.
By {\em all} $q$ we mean all $q$ that are not a root of unity. Saying {\em almost all}  we
exclude a finite set of $q$ distinct from $1$.
We say that a property holds for {\em generic} $q$ if it is true for  almost all $q$
when restricted to every finite-dimensional subspace.
That can be formalized by passing to the local ring of rational functions regular at $q=1$.

\subsection{Extremal twist and complete reducibility of tensor products}
Let us  recall the construction of extremal twist, which is responsible for irreducible decomposition of tensor product
 $V\tp Z$ of two highest weight modules \cite{M3}.

Consider an irreducible highest weight  $U_q(\g)$-module $Z$ as a  $U_q(\g_-)$-module and denote by  $I^-_Z$ the left
ideal in $U_q(\g_-)$ that annihilates the highest vector. Denote by  $I^+_Z\subset U_q(\g_+)$
the left ideal $\si(I^-_Z)$, where $\si$ is an involutive automorphism of $U_q(\g)$ defined on the generators by
the assignment $e_\al\to f_\al$, $f_\al\to e_\al$,  $q^{\pm h_\al}\to q^{\mp h_\al}$ for all $\al \in \Rm^+$.

Now let $V$ and $Z$ be a pair of irreducible $U_q(\g)$-modules of highest weights $\nu$ and $\zt$,
respectively, and highest vectors $1_\nu$ and $1_\zt$. Denote by $V^+_Z$ the kernel of the left ideal $I^+_Z$ in $V$ and
similarly define $Z^+_V$. There are linear isomorphisms between $V^+_Z$, $Z^+_V$, and the subspace $(V\tp Z)^+$ spanned by
singular vectors in $V\tp Z$. To describe this correspondence, present $u\in (V\tp Z)^+$ as $u=1_\nu\tp z_{\mu-\nu}+\ldots + v_{\mu-\zt}\tp 1_\zt$,
where the terms with factors of other weights are suppressed. The isomorphisms are given by the assignments $\bar \dt_l(u)=v_{\nu-\zt}$
and
$\bar \dt_r(u)=v_{\mu-\nu}$. We denote by $\dt_l$ and $\dt_r$ their inverse isomorphisms.

Let $\omega$ denote  an  anti-algebra involution on $U_q(\g)$
defined as $\gm^{-1}\circ \si$. This map preserves the comultiplication on $U_q(\g)$.
Choose a weight basis in $V$ and present a singular vector $u$ as $u=\sum_i v_i\tp f_i 1_\zt$ for some $\{f_i\}\subset U_q(\g_-)$.
With $v=\bar \dt_l(u)$, define a map $\theta_{V,Z}(v)\colon V\to V/\omega(I^+_Z)V$ sending
 $v$ to the image of $\sum_i\gm^{-1}(f_i)(v_i)$ in  $V/\omega(I^+_Z)V$.
This map is independent of the choice of $f_i$ in the  presentation of $u$ and called extremal twist.

Define a canonical bilinear  form on $V\tp Z$ as the product of the $\omega$-contravariant  forms on $V$ and $Z$
(such forms do exist, are unique up to a scalar multiplier and are non-degenerate).
As $\omega$ is a coalgebra map, the canonical form  is also contravariant.
If $\langle \>.\>,\>.\>\rangle$ is  the contravariant form on $V$, then the form
$\langle \theta_{V,Z}(\>.\>),\>.\>\rangle$ is the pullback of the canonical form
via the isomorphism $\delta_l\colon V^+_Z\to (V\tp Z)^+$.
Similarly one can consider   $\theta_{V,Z}$ instead of $\theta_{Z,V}$ due to the obvious symmetry between $V$ and $Z$.

\begin{thm}[\cite{M3}]
\label{com_red_crit}
The following are equivalent
\begin{enumerate}
  \item[i)]  $V\tp Z$ is completely reducible.
  \item[ii)] $\theta_{V,Z}$ is bijective.
  \item[iii)] All highest weight submodules in $V\tp Z$ are irreducible.
  \item[iv)] $V\tp Z$ is the sum of submodules of highest weight.
\end{enumerate}
\end{thm}
\noindent
Practically we use  ii) and iii) in order to check i); the criterion iv) is included for completeness,

As an example, consider the case whith $\g=\s\l(2)$ and $Z$  a Verma module of highest weight $\la$.
Choose a vector $v=f^m_\al 1_\nu\in V$ of weight $\xi=\nu-m \al$.
It is easy to check that
\be
\theta_{V,Z}(v)\propto \prod_{k=1}^{m}\frac{[(\la+\rho+\xi,\al^\vee)+k]_{q_\al}}{[(\la+\rho,\al^\vee)-k]_{q_\al}}v,
\label{theta_sl2}
\ee
This can be done via presenting $\dt_l(v)$ as $\Fc(v\tp 1_\zt)$, where $\Fc=1\tp 1+\ldots \in U_q(\g_+)\hat \tp U_q(\g_-)$
is a lift of the (unique)
invariant element ${}^*\!1_{-\zt}\tp 1_\zt+ \ldots \in {}^*\!Z\hat\tp Z$. This lift can be easily computed
for the $\s\l(2)$-case. For $Z$ a Verma module, $\theta_{V,Z}$ coincides with $\gm^{-1}(\Fc_2)\Fc_1$ represented in $\End(V)$. The latter can be expressed through
the operators $\pi(s)$, then (\ref{Zhb_fact}) translates to (\ref{theta_sl2}).
\subsection{Description of base  module  and its generalized extremal spaces}
The base module $M$ for quantum $\Sbb^{2n}$ has a PBW basis that makes it isomorphic to the vector space of polynomials in $n$ variables.
To describe this basis, we need root vectors of weights $-\ve_i$ for $i=1,\ldots,n$.
Define
$e_{\ve_{1}}=e_{\al_{1}}$ and  $f_{\ve_{1}}=f_{\al_{1}}$ and proceed recurrently for $i>1$ as
$$
e_{\ve_{i+1}}=[e_{\al_{i+1}},e_{\ve_i}]_{q}, \quad
f_{\ve_{i+1}}=[f_{\ve_i}, f_{\al_{i+1}}]_{\bar q}
.
$$
Fix $\la\in \h^*$ by the condition $q^{2(\la,\ve_i)}=-q^{-1}$ for all $i=1,\ldots, n$  and $(\al_i, \la)=0$ for $i>1$.
We call it base weight.
The Verma module  $\hat M_\la$ of highest weight $\la$
with the canonical generator $1_\la$ has singular vectors $f_{\al_i}1_\la$ with  $i>1$ and $\hat f_{\dt}1_\la$.
Define  $M$ as the quotient of $\hat M_\la$ by the sum of submodules generated by these vectors.
As proved in \cite{M1},  it  has a basis of weight vectors $ f_{\ve_1}^{m_1}\ldots f_{\ve_{n}}^{m_{n}}1_\la$ where $m_i$ take all
possible values in $\Z_+$ (we use the same notation for the image of $1_\la$ in $M$). In particular,
$\Char(M)=\prod_{i=1}^{n}(1-q^{-\ve_i})^{-1}q^\la$.

\begin{lemma}
The quantum group $U_q(\g)$ acts on $M$ by
\be
e_{\al_i} f_{\ve_1}^{m_1}\ldots f_{\ve_{n}}^{m_{n}}1_\la
&\propto &f_{\ve_1}^{m_1}\ldots f_{\ve_{i-1}}^{m_{i-1}+1}f_{\ve_i}^{m_i-1} \ldots f_{\ve_{n}}^{m_{n}}1_\la,
\nn\\
f_{\al_i} f_{\ve_1}^{m_1}\ldots f_{\ve_n}^{m_n}1_\la
&\propto &
f_{\ve_1}^{m_1}\ldots f_{\ve_{i-1}}^{m_{i-1}-1} f_{\ve_{i}}^{m_{i}+1}
\ldots f_{\ve_n}^{m_n}1_\la,\quad i=1,\ldots, n,
\nn
\ee
where the suppressed numerical factors are non-zero once  $q$ is not a root of unity.
\label{action on M}
\end{lemma}
\begin{proof}
  By weight arguments the action has the specified form up to numerical factors. That they are not zero can be
  proved as follows. First of all, they are given explicitly for the $f_{\al_i}$-action in
  Lemma 3.2 in \cite{M1}.  One can check that $f_{\ve_{i}}$  satisfy the relation
  $f_{\ve_{i+1}}f_{\ve_{i}}= q^{-1}f_{\ve_{i}}f_{\ve_{i+1}}$ as operators in  $\End(M)$.
  Using these relations, one can easily prove that  the factors arising from the $e_{\al_i}$-action are non-zero too.
\end{proof}
\noindent
Remark that irreducibility of $M$ readily follows from these formulas since all weights in $M$ are multiplicity free,
and $M$ clearly has no singular vectors. Alternatively, that can be proved via the contravariant  form
on $M$, which  is found to be non-degenerate in \cite{M1}.

Denote by  $\rho\in \h^*$  the half-sum of positive roots and put $\al^\vee=\frac{2\al}{(\al,\al)}$ for all $\al \in \h^*$.
Fix a finite-dimensional $U_q(\g)$-module $V$  of highest weight $\nu$ and introduce $\ell_i=(\nu+\rho,\al_i^\vee)-1\in \Z_+$
for $i=1,\ldots, n$.
It is known that $V$  is a quotient of the Verma module $\hat M_\nu$ by the submodule $I^-_V1_\nu$, where $I^-_V$ is a left
ideal in $U_q(\g_-)$ generated by $\{f_{\al_i}^{\ell_i+1}\}_{i=1}^n$, \cite{Jan}. Set $I^+_V=\si(I^-_V)$. Then
singular vectors in $V\tp M$ are parameterized by $M^+_V=\ker I^+_V\subset M$, as explained above.

\begin{propn}
For all $V$, the contravariant form on $M$ is non-degenerate when restricted to $M^+_V$.
The module $M$ splits into the orthogonal sum $M=M^+_V\op \omega(I^+_V)M$
with
$$
M^+_V =\Span\{ f_{\ve_1}^{m_1}\ldots f_{\ve_{n}}^{m_{n}}1_\la\}_{m_1\leqslant \ell_1, \ldots, m_n\leqslant \ell_n },\quad
\omega(I^+_V)M =\Span\{ f_{\ve_1}^{k_1}\ldots f_{\ve_{n}}^{k_{n}}1_\la\}_{k_1, \ldots, k_n},
$$
where  $k_i>\ell_i$ for some $i=1,\ldots,n$.
\label{non-degen}
\end{propn}
\begin{proof}
Observe that $\omega(I^+_V)M=\sum_{i=1}^{n}f_{\al_i}^{\ell_i+1}M$, which proves the right equality. The
left equality follows from Lemma \ref{action on M}. Since $\omega(I^+_V)M$ is orthogonal to $M^+_V$ and the form
is non-degenerate due to irreducibility of $M$, the proposition is proved.
\end{proof}


Singular vectors in $V\tp M$ can be alternatively parameterized by the subspace $V^+_M\subset V$, which is the joint
kernel of $\{e_{\al_i}\}_{i=2}^n$ and $\hat e_\dt$.
The weight subspaces in $V^+_Z$ have dimension $1$ and carry  weights
$\nu-\sum_{i=1}^{n}m_i\ve_i$, with $0\leqslant m_i\leqslant \ell_i$. In the classical limit,
$V^+_M$ is spanned by  highest vectors of irreducible $\k$-modules, according to the Gelfand-Zeitlin reduction.

Consider, for example, the case of $\g=\s\o(5)$ and the $U_q(\g)$-module $V$ of highest weight $\nu=3\nu_1+2\nu_2$, where
$\nu_i$ are fundamental weights, $(\nu_i, \al_j^\vee)=\dt_{ij}$, $i,j=1,2$. The subspace $M^+_V$ is spanned
by $f_{\ve_1}^{m_1}f_{\ve_2}^{m_2}1_\la$ with $m_1\leqslant 3$ and $m_2\leqslant 2$. Its reciprocal space $V^+_M$
is spanned  by vectors of weights $\nu-m_1\ve_1-m_2\ve_2$ of multiplicity $1$, which are depicted on the
weight lattice of $V$ by the fat circles:
$$
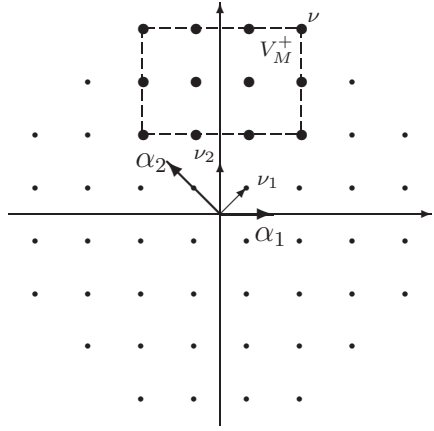
\begin{picture}(180,180)
\put(20,100){\vector(1,0){160}}
\put(115,160){$\scriptstyle{V^+_M} $}
\put(100,20){\vector(0,1){160}}
\multiput(71,170)(20,0){4}{\circle*{4}}
\multiput(71,150)(20,0){4}{\circle*{4}}
\multiput(71,130)(20,0){4}{\circle*{4}}
\multiput(70,110)(20,0){4}{\circle*{2}}
\multiput(70,90)(20,0){4}{\circle*{2}}
\multiput(70,70)(20,0){4}{\circle*{2}}
\multiput(70,50)(20,0){4}{\circle*{2}}
\multiput(70,30)(20,0){4}{\circle*{2}}
\multiput(50,150)(0,-20){6}{\circle*{2}}
\multiput(30,130)(0,-20){4}{\circle*{2}}
\multiput(150,150)(0,-20){6}{\circle*{2}}
\multiput(170,130)(0,-20){4}{\circle*{2}}
\put(100,100){\vector(1,1){10}}
\put(100,100){\vector(0,1){20}}
\multiput(71,170.5)(6,0){10}{\line(1,0){4}}
\multiput(71,130.5)(6,0){10}{\line(1,0){4}}
\multiput(70.5,131)(0,6){7}{\line(0,1){4}}
\multiput(130.5,131)(0,6){7}{\line(0,1){4}}
\thicklines
\put(100,100){\vector(-1,1){20}}
\put(100,100){\vector(1,0){20}}
\put(113,90){$\al_1$}
\put(68,118){$\al_2$}
\put(114,111){$\scriptstyle{\nu_1}$}
\put(90,121){$\scriptstyle{\nu_2}$}
\put(133,173){$\scriptstyle{\nu}$}
\end{picture}
$$
An interesting problem is evaluation of the quantum reduction algebra of the pair $(\g,\k)$ or, more
exactly, of the left ideal $U_q(\g)I^+_M$. Its positive part is generated by a pair of  commuting
elements $z_{\ve_1}=e_{\ve_1}[h_{\al_2}]_q - q e_{\ve_2} f_{\al_2}$ and $z_{\ve_2}=e_{\ve_2}$. It allows to restore $V^+_M$ from
the vector of lowest weight in  $V^+_M$. Negative generators, which are unknown to us,  would deliver $V^+_M$ out of the highest
vector of $V$.

The subspace $V^+_M$ is spanned by ``singular vectors" of a ``subalgebra" $U_q(\k)\subset U_q(\g)$, which
is also unknown (and whose existence is under question). We believe that such a subalgebra does exist, maybe in an appropriate extension of $U_q(\g)$. One indication in favor of this conjecture is the presence of a coideal subalgebra, which is a quantization of the stabilizer
of a different point in $\Sbb^{2n}$, cf. Section \ref{SecSymPair}.

\subsection{Pseudo-parabolic Verma modules}
Suppose that  $\k\subset \g$ is a reductive subalgebra of maximal rank. Suppose that
its polarization is induced by the polarization of $\g$, so there are inclusions
 $\Rm^\pm_\k\subset \Rm^\pm_\g$ of their root subsystems.
Take a weight $\xi\in \h^*$ that is integral dominant with respect to $\k$
and let $X$ be the irreducible finite-dimensional $\k$-module  $X$ of highest weight  $\xi$.
For each $\al\in \Pi^+_\k$  the sum of $\xi$ and the base weight  $\la$  satisfies the condition
\be
q^{2(\xi+\la+\rho,\al)-m_\al(\al,\al)}=1,
\label{KazhKac}
\ee
where  $m_\al=(\xi,\al^\vee)+1$ a positive integer.
Then there is a  submodule $\hat M_{\xi+\la-m_\al\al}$ in the Verma module $\hat M_{\xi+\la}$, \cite{DCK}.
Let $M_{X,\la}$ denote the quotient of $\hat M_{\xi+\la}$ by the sum of $\hat M_{\xi+\la-m_\al\al}$ over all $\al\in \Pi^+_\k$.
If $\k$ were a Levi subalgebra in $\g$, that is, if $\Pi^+_\k\subset \Pi^+_\g$, then $M_{X,\la}$ would be a parabolic Verma module.
That justifies the name of pseudo-parabolic Verma module for $M_{X,\la}$ when $\Pi^+_\k\not \subset \Pi^+_\g$.

For non-Levi $\k$,  $M_{X,\la}$ is factored through a parabolic module $M_{\check X_\la}$
relative to the maximal  Levi subalgebra $\l$ among those sitting inside $\k$.
The finite-dimensional $U_q(\l)$-module $\check X_\la =\check X\tp \C_\la$ is of the same highest weight $\xi+\la$.
 As a $U_q(\l)$-module,  $M_{\check X_\la}$ is isomorphic to $U_q(\n_-)\tp \check X_\la$, where $U_q(\n_-)\subset U_q(\g)$
 is a subalgebra  invariant under the
 adjoint action of $U_q(\l)$. It is a deformation of $U(\n_-)$, where $\n_-=\g_-\ominus \l_-$, see Section \ref{loc_fin_parabolic} for details.

In the case of concern,
$\l=\g\l(n)$, and $M_{X,\la}$ is a quotient of $M_{\check X_\la}$ by the submodule generated by
the image of  $\pi_{\al_1}(2h_\al+1)f_\dt^m1_{\xi+\la}$ with $m=m_\dt$. By Proposition \ref{singular}, it
 is singular  unless it is zero.
\begin{lemma}
 For all $m\in \Z_+$, the vector $\pi_{\al_1}(2h_\al+1)f_\dt^m1_{\xi+\la}$  does not vanish in $M_{\check X_\la}$.
\end{lemma}
\begin{proof}
  We can assume that $n=2$, then $\hat f_\dt$ is central in $U_q(\g_-)$.
  It is easy to check that the ordered monomials in $f_{\al_1},\hat f_\dt, f_{\al_1+\al_2},f_{\al_2}$
  whose degree of $f_{\al_2}$ is less than $m_{\al_2}$
  deliver a basis in $M_{\check X_\la}$.
  Factorization (\ref{factorize f_dt}) implies
$$
\pi_{\al_1}(2h_\al+1)f_\dt^m1_{\xi+\la}=q^{-m}\prod_{i=0}^{m-1}\hat f_\dt(s+2i)1_{\xi+\la}=q^{-m}\hat f_{m\dt}(s)1_{\xi+\la},
$$
  where $s=2(\xi+\la,\al)-2m+2$ and $\hat f_{m\dt}(s)\in U_q(\g_-)$. Observe that $\hat f_{m\dt}(s)$ is regular as $q^{2(\la,\al_1)}=-q^{-1}$.
  Induction on $m$ proves that the right-hand side expands as $cf_\dt^m1_{\xi+\la}+\ldots$,
  where  $c\not=0$ and the omitted terms involve different basis elements. This proves the statement.
\end{proof}
Thus $M_{X,\la}$ is a quotient of  $M_{\check X_\la}$ by a proper submodule.
Since $M_{\check X_\la}$ is locally finite over $U_q(\l)$ by Proposition \ref{loc_fin_parabolic}, so is  $M_{X,\la}$.

\begin{corollary}
  For generic $q$, $\Char(M_{X,\la})\leqslant \Char(M)\Char(X)$.
\label{char-pseudo-parabolic}
\end{corollary}
\begin{proof} In the classical limit $q\to 1$,   the  vector $\hat f_{m\dt}(s)$ goes over to
 $f_\dt^m$, and
  the $U_q(\g_-)$-module $M_{X,\la}$ turns into the induced module
  $U(\g_-)\tp_{U(\k_-)}X$. Therefore $\Char (M_{X,\la})\leqslant \Char(M)\Char(X)$
  for generic $q$,
  since deformation does not increase quotients.
\end{proof}

\subsection{Pseudo-parabolic category $\O(\Sbb^{2n})$}
In this section we prove that generalized parabolic modules of highest weight are irreducible, and tensor products
of the base module with finite-dimensional modules are completely reducible. We end up showing that
the corresponding subcategory in the category $\O$ is equivalent to the category of finite-dimensional $\k$-modules.

Fix a finite-dimensional $U_q(\g)$-module  $V$.  The subspace $V^+_M\subset V$ is the kernel of the left ideal $ I^+_{M}\subset U_q(\g_+)$  generated by $e_{\al_i}$, $i>1$, and by $\hat e_\dt(s)$,
with $s=2(\la,\al_1)-1$ (note that $\hat e_\dt(s)\propto \hat e_\dt \mod U_q(\g_+) e_{\al_2}$ for such $s$). In the classical limit, $V^+_M$  is the kernel of $\k_+$.

\begin{propn}
\label{homomorphism}
    For each non-zero $v\in V^{+}_M[\xi]$, the homomorphism
  $\hat M_{\xi+\la}\to  V\tp M$, $1_{\xi+\la}\mapsto \delta_l(v)$, factors through a homomorphism
  $M_{X,\la}\to V\tp M$.
\end{propn}
\begin{proof}
 Since $M$ is irreducible, the right dual $N$ to $M$ is a module of lowest weight, cf. Section B.
 It is a quotient  of the lower Verma module by the submodule generated by  $\hat e_\dt1_{-\la}$ and $e_{\al}1_{-\la}$, $\al\in \Pi^+_\l$.
 Furthermore, the singular vector $\hat f_{m\dt}(x)1_{\xi+\la}\in \hat M_{\xi+\la}$ with $m=m_\dt$ and $x=2(\xi+\la,\al)-2m+1$
 vanishes in $M_{X,\la}$. Therefore, by Corollary \ref{high-low}, $\Hom(M_{X,\la},V\tp M)$ is isomorphic to $V^+_M[\xi] \cap \ker \hat f_{m\dt}(x)$.
 The statement will be proved if we show that $V^+_M[\xi] \subset \ker \hat f_{m\dt}(x)$.

The operator $\pi_\al(x+2m)$ is invertible on $V$ by Proposition \ref{pi-invertible}, so put $v_0=\pi_\al^{-1}(x+2m)v$.
Similarly, the operator $\pi_\al(2h_\al+s)$, with $s=2(\la,\al)-1$ is invertible on $V$.
Then, by (\ref{e-pi-int}),
$$0=\hat e_\dt(s)\pi_\al(x+2m)v_0=
\hat e_\dt(s)\pi_\al(2h_\al+s+2)v_0
=
q^{-1}\pi_\al(2h_\al+s) e_\dt v_0,
$$
hence $e_\dt v_0=0$. The vector $v_0$ generates a $U_q(\g^{(\dt)})$-submodule with $f^m_\dt v_0=0$.
Then, by (\ref{f-pi-int}),
$\hat f_{m\dt}(x)\pi_\al(x+2m)v_0=
\pi_\al(x) f_\dt^{m}v_0=0$ as required.
\end{proof}

Denote by $M'_{X_i,\la}$ the image of the module $M_{X_i,\la}$ in $V\tp M$, then $\Char(M'_{X_i,\la})\leqslant \Char(M_{X_i,\la})$.
We aim to prove that $V\tp M$ is a direct sum of  $M'_{X_i,\la}$ and essentially $M'_{X_i,\la}=M_{X_i,\la}$.
The proof will be done by induction on the rank of $\g$.
For each $k=1,\ldots, n$, we identify the Lie subalgebra $\g^{(k)}=\s\o(2k+1)\subset \g$ whose root basis is $\{\al_1,\ldots, \al_k\}$. The base module $M$
contains base modules for $\Sbb^{2k}$, which we denote by $M^{(k)}$. Obviously, $M^{(k)}\subset M^{(k+1)}$ for all $k<n$.

\begin{thm}
\label{thm-dir-sum}
Let  $V$ be a finite-dimensional $U_q(\g)$-module. Then
\begin{enumerate}
  \item   $V\tp M$ is completely reducible for all $q$ and
  splits into a direct sum $\op_{i}M'_{X_i,\la}$ for
  an irreducible decomposition $V=\op_i X_i$ over $\k$,
 \item all modules $M'_{X_i,\la}$ are irreducible and $\Char(M'_{X_i,\la})=\Char(M)\Char(X_i)$ for all $q$,
 \item $M_{X,\la}\simeq M'_{X,\la}$ for generic $q$.
\end{enumerate}
\end{thm}
Note that irreducibility of $M'_{X_i,\la}$ is equivalent to complete reducibility of $V\tp M$,
by Theorem \ref{com_red_crit}, iii). Below we give a full proof only for generic $q$.
For all $q$, complete reducibility of
$V\tp M$ follows from a direct calculation of $\theta_{M,V}$ that employs a relation of $\theta_{M,V}$ with extremal projectors,
in generalization of (\ref{theta_sl2}). That is based on a more advanced development of \cite{M3} and done in \cite{M5}.
\begin{proof}

  Consider the case of $n=1$. Then $\g\simeq \s\l(2)$, $\k=\h$, and $M$ is a Verma module.
  Let $\nu$ be the highest weight of $V$. The extremal twist acts on a vector $f_{\ve_1}^m1_\nu\in V$ of weight $\xi=\nu-m\ve_1$ as in (\ref{theta_sl2}), where $\al=\ve_1$.
It does not vanish once $q^{2(\la,\ve_1)}=-q^{-1}$, therefore  $V\tp M=\op_iM'_{X_i,\la}$, where all $M'_{X_i,\la}$ are irreducible
Verma modules.
This proves the theorem for  $n=1$ and all $q$.

Suppose that the theorem is proved for $n\geqslant 1$.
An irreducible  $U_q(\g^{(n+1)})$-module $V$ is parameterized by the highest weight $\nu=l_1\ve_1+\ldots +l_{n+1}\ve_{n+1}$ with
$0\leqslant l_i\leqslant l_{i+1}$.
Observe through the Gelfand-Zeitlin reduction that all irreducible $U_q(\g^{(n)})$-modules can be obtained by considering the special
case of $l_n=l_{n+1}$. This corresponds to the zero $n+1$-th coordinate in the expansion of $\nu$ over the fundamental weights. Then  $V^+_M$ is contained in the $U_q(\g^{(n)})$-submodule $V^{(n)}$ generated by the highest vector
of $V$, and
the corresponding singular vector lies in $V^{(n)}\tp M^{(n)}$. By the assumption, it is completely reducible over $U_q(\g^{(n)})$, therefore $\theta_{V^{(n)},M^{(n)}}$ is bijective. Since the extremal twist is independent of the choice of lift for singular vectors, $\theta_{V^{(n)},M^{(n)}}=\theta_{V,M}$, and thus
$\theta_{V,M}$ is bijective. Therefore $V\tp M$ is completely reducible for special $V$ (and for all $q$).

Furthermore, Corollary \ref{char-pseudo-parabolic} implies
$$
\Char(V\tp M)=\sum_{i}\Char(M'_{X_i,\la})\leqslant \sum_{i}\Char(M_{X_i,\la})\leqslant \sum_{i}\Char(M)\Char(X_i)=\Char(M)\Char(V)
$$
for generic $q$.
Since $\Char(V)\Char(M)=\Char(V\tp M)$, we conclude that
$$
\sum_{i}\Char(M'_{X_i,\la})=\sum_{i}\Char(M_{X_i,\la})=\sum_{i}\Char(M)\Char(X_i)
$$
and hence $M'_{X_i,\la}=M_{X_i,\la}$ for all $i$. Since $M_{X_i,\la}'$ are
rational submodules in $V\tp M$, their characters do not exceed their values at generic pont.
But they cannot be less than $\Char(M)\Char(X_i)$ either because they sum up to $\Char(V)\Char(M)$.
This proves 3) and also proves 2) for all $q$ and special $V$.

Now suppose that $V$ is arbitrary. Proposition \ref{homomorphism} implies that its all highest weight submodules in
$V\tp M$ are irreducible, for generic $q$. This  proves 1), by Theorem \ref{com_red_crit}, iii).
\end{proof}
Let us illustrate the induction transition in the proof with the upper part of the weight diagram of the module
$V=(3,2)$ for  $\g=\s\o(5)$. The weights of $V^+_M$  lie on the three horizontal
dashed lines.
$$
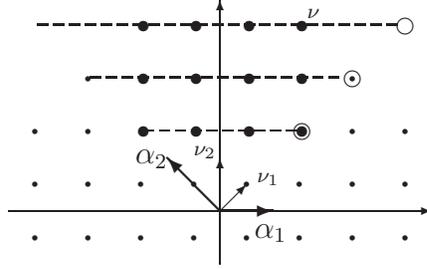
\begin{picture}(180,100)
\put(20,20){\vector(1,0){160}}
\put(100,0){\vector(0,1){100}}
\multiput(71,90)(20,0){4}{\circle*{4}}
\multiput(71,70)(20,0){4}{\circle*{4}}
\multiput(71,50)(20,0){4}{\circle*{4}}
\multiput(70,30)(20,0){4}{\circle*{2}}
\multiput(70,10)(20,0){4}{\circle*{2}}
\put(131,50){\circle{6}}\put(150,70){\circle{6}}\put(170,90){\circle{6}}
\multiput(50,70)(0,-20){4}{\circle*{2}}
\multiput(30,50)(0,-20){3}{\circle*{2}}
\multiput(150,70)(0,-20){4}{\circle*{2}}
\multiput(170,50)(0,-20){3}{\circle*{2}}
\put(100,20){\vector(1,1){10}}
\put(100,20){\vector(0,1){20}}
\multiput(31,90.5)(6,0){23}{\line(1,0){4}}
\multiput(51,70.5)(6,0){16}{\line(1,0){4}}
\multiput(71,50.5)(6,0){10}{\line(1,0){4}}
\thicklines
\put(100,20){\vector(-1,1){20}}
\put(100,20){\vector(1,0){20}}
\put(113,10){$\al_1$}
\put(68,38){$\al_2$}
\put(114,31){$\scriptstyle{\nu_1}$}
\put(90,41){$\scriptstyle{\nu_2}$}
\put(133,93){$\scriptstyle{\nu}$}
\end{picture}
$$
These lines can be obtained from three special modules of highest weights  $\ell(\ve_1+\ve_2)$ with $\ell=3,5,7$.
They are
marked on the diagram with large circles. Their weight subspaces are in $U_q(\g^{(1)})$-submodules $V^{(1)}$ of dimensions $4$, $6$, $8$, therefore  we can restrict to $U_q(\g^{(1)})$
when calculating singular vectors.

\begin{conjecture}
\em
  For any irreducible $\k$-module $X$, the module $M_{X,\la}$ is isomorphic to $M'_{X,\la}$  for all $q$.
\end{conjecture}
\noindent
Let us comment that its proof  is equivalent to computing $\Char(M_{X,\la})$.
In the light of Theorem \ref{thm-dir-sum}, 3), the inequality of Corollary \ref{char-pseudo-parabolic}
becomes equality at generic $q$ and can be rectified as $\Char(M_{X,\la})\geqslant \Char(M)\Char(X)$ for all $q$.

Recall that the category $\O$ is formed by finitely generated $U_q(\h)$-diagonalizable
$U_q(\g)$-modules that are locally finite over $U_q(\g_+)$.
Define pseudo-parabolic category $\O(\Sbb_q^{2n})$ as the full subcategory in $\O$ whose objects are submodules in $V\tp M$, where
$V\in \Fin_q(\g)$.
Theorem \ref{thm-dir-sum} can be reworded as
\begin{propn}
  $\O(\Sbb_q^{2n})$ is a semi-simple Abelian category equivalent to $\Fin(\k)$.
  It is a $\Fin_q(\g)$-module category generated by $M$.
  \label{quasi-parabolic category}
\end{propn}
\begin{proof}
  The category $\O(\Sbb_q^{2n})$ is additive semi-simple, hence it is Abelian.
  By construction, it is a module category over $\Fin_q(\g)$.

  The functor  $\O(\Sbb_q^{2n})\to \Fin(\k)$ is constructed as follows.
  Given a module $M_1$ from $\O(\Sbb_q^{2n})$ define  $\check X_\la$ to be
  its finite-dimensional $U_q(\l)$-submodule generated by the span of singular vectors in $M_1$. Let $\check  X_0 $ be
   the classical limit of $\check  X_\la\tp \C_{-\la}$ equipped with the corresponding action of $U(\l)$. Proceed
    to the $\k$-module parabolically induced  from $\check  X_0$ and then
  to its finite-dimensional quotient $X$ by the maximal submodule that has zero intersection with $\check  X_0$. These steps, apart from maybe the last one, respect morphisms. But morphisms of parabolic $\k$-modules induced from the $\l$-modules descent to morphisms of their finite-dimensional quotients. Now let $M_2$ be another module from $\O(\Sbb_q^{2n})$ and define $\check Y_\la$, $\check Y_0$, and $Y$
  accordingly.  We have
$$
  \Hom_{U_q(\h)}(M_1^+,M_2^+)=\Hom_{U_q(\h)}(\check  X_\la^+,\check  Y_\la^+)\simeq \Hom_{U(\h)}(\check  X_0^+,\check  Y_0^+)= \Hom_{U(\h)}(X^+,Y^+),
$$
where $+$ designates span of singular vectors.
These imply the isomorphisms
$$
  \Hom_{U_q(\g)}(M_1,M_2)\simeq\Hom_{U_q(\l)}(\check  X_\la,\check  Y_\la)\simeq \Hom_{U(\l)}(\check  X_0,\check  Y_0)\simeq \Hom_{U(\k)}(X,Y).
$$
The middle isomorphism is due to \cite{EK}, while the other two are consequences
of semi-simplicity.

Finally, every module from $\Fin(\k)$ is isomorphic to one constructed this way because it can be realized as a submodule of a module
    from $\Fin(\g)$.
\end{proof}
In the next section we consider two other $\Fin_q(\g)$-module categories:
equivariant projective modules over $\C_q[\Sbb^{2n}]$ and representations
of a coideal subalgebra in $U_q(\g)$.
These categories will be shown  equivalent  to $\O(\Sbb_q^{2n})$.

\section{Equivariant vector bundles over quantum spheres}
\label{SecEqVB}

In some respects, this section is analogous to the corresponding section of \cite{M4}, so we refrain
from giving detailed proofs where they repeat similar arguments of the  proofs therein.
For the sake of compatibility with \cite{M2}, we change the comultiplication in $U_q(\g)$ to
$$\Delta(f_\al)= f_\al\tp q^{-h_\al}+1\tp f_\al,\quad\Delta(q^{\pm h_\al})=q^{\pm h_\al}\tp q^{\pm h_\al},\quad\Delta(e_\al)= e_\al\tp 1+q^{h_\al}\tp e_\al, \quad \al \in \Pi^+.$$
Since the two coproducts are conjugated via an R-matrix, this modification does not affect the conclusions of the previous sections.

In the deformation context, we pass to the $\C[[\hbar]]$-extension of $U_q(\g)$ completed in the $\hbar$-adic topology, which
we denote by $U_\hbar(\g)$. Note that  the $\C[[\hbar]]$-extension of the base module
does not support the action of $U_\hbar(\g)$. Nevertheless, the $U_\hbar(\g)$-action on $\End(M)$ is well defined.

Given a $\C$-algebra $A$ we call a $\C[[\hbar]]$-algebra $A_\hbar$ quantization of $A$ if it is a free $\C[[\hbar]]$-module
and $A_\hbar/ \hbar A_\hbar$ is isomorphic to  $A$ as a $\C$-algebra.
  The quantization is called equivariant if $A_\hbar$ is a $U_\hbar(\g)$-module algebra such that the $U_\hbar(\g)$-action is a deformation of a $U(\g)$-action on $A$.
If a $U_\hbar(\g)$-algebra $A_\hbar$ is a $\C[[\hbar]]$-extension of a $U_q(\g)$-algebra $A_q$, then $A_q$ is also called (equivariant) quantization of $A$. This should not cause a confusion as the ring of scalars is always clear from the context.

Suppose that $\Gamma$ is a projective $A$-module with an $U(\g)$-action such that the  multiplication $A\tp \Gamma\to \Gamma$ is
equivariant. Suppose that $A_\hbar$ is an equivariant quantization of $A$. We call a projective left $A_\hbar$-module $\Gamma_\hbar$  quantization of $A$ if $\Gamma_\hbar$ is also a $U_\hbar(\g)$-module such that the multiplication
$A_\hbar \tp \Gamma_\hbar\to \Gamma_\hbar$ is equivariant and the action of $U_\hbar(\g)$ is a deformation of the action of $U(\g)$.
Equally we consider the version of right $A_\hbar$-modules $\Gamma_\hbar$. Our convention about indexing with $q$ instead of
$\hbar$ applies to $A_q$-modules as well.

By the Serre-Swan theorem \cite{S,Sw}, finitely generated projective $A$-modules are sections of  vector bundles in the classical algebraic geometry. Such modules over $A_\hbar$ can be regarded as quantized vector bundles.

\subsection{Projective modules over $\C_q[\Sbb^{2n}]$}

By a classical equivariant vector bundle over $\Sbb^{2n}$ with fiber $X$ we understand the (left or right)
finitely generated projective $\C[\Sbb^{2n}]$-module $\Gamma(\Sbb^{2n},X)$ of its global sections. It can be  realized as the subspace of $\k$-invariants in
$\C[G]\tp X$, where $G$  is either $SO(2n+1)$ or its simply connected  covering if $X$ is a spin representation of $\k$.

The quantum polynomial algebra  $\C_q[\Sbb^{2n}]$  is represented as a subalgebra $\Ac\subset \End(M)$, cf. \cite{M2}.
In accordance with the above definitions, quantization of an equivariant vector bundle on $\Sbb^{2n}$ is a $U_q(\g)$-equivariant deformation of
$\Gamma(\Sbb^{2n},X)$ in the class of right $\C_q[\Sbb^{2n}]$-modules. It is realized as $\hat P (V\tp \Ac)$,
where $\hat P\in \End(V)\tp \Ac$ is a $\U_q(\g)$-invariant idempotent.
Such idempotents can be  constructed via  decomposition of $V\tp M$ due to the following fact.

\begin{propn}
Every invariant map $V\tp M\to W\tp M$ belongs to  $\Hom_\C(V,W)\tp \Ac$.
\label{Kostant}
\end{propn}
\noindent
One can prove it using the fact that $\Ac$ exhausts all of the locally finite part of the $U_q(\g)$-module $\End(M)$,
which can be checked with the help of Corollary \ref{high-low}.
This is the answer to a Kostant's problem for quantum groups, \cite{KST}.

For each weight vector $v\in V^{+}_M$ there is a projector  $\hat P\in \End(V)\tp \Ac$ to a particular copy of $M_{X,\la}'$
generated by the singular vector $\dt_l(v)\in V\tp M$.
\begin{thm}
\label{hom-bund1}
The right $\Ac$-module $\hat P(V\tp \Ac)$  is a quantization of $\Gamma(\Sbb^{2n},X)$.
\end{thm}
\begin{proof}
We can assume that $X$ is irreducible of highest weight $\xi$.
The $U_q(\g)$-module $\hat P(V\tp \Ac)$ is isomorphic to the locally finite part of $\Hom_\C(M,M_{X,\la})$.
Let us prove that it has a similar $U_q(\g)$-module structure
 as the $U(\g)$-module $\Gamma(\Sbb^{2n},X)$.

  First of all, observe that the preceding considerations can be conducted in the opposite category $\O$ of modules whose weights
  are bounded from below, leading to similar conclusions with regard to semi-simplicity {\em etc}.
  Denote by $N_{X,\la}$ the module of lowest weight  $-\xi-\la$
  that is opposite to $M_{X,\la}$. The annihilator of its lowest vector is obtained from $I^-_{M_{X,\la}}$ by the involution $\si$.
  Let $N'_{X,\la}\simeq (M_{X,\la}')^*$ stand for its irreducible image in $V^*\tp N$.

    Since $M_{X,\la}'$ and $M$ are  irreducible,
  the vector space $\Hom\bigl(W,\Hom_\C(M,M_{X,\la}')\bigr)$ is isomorphic to $\Hom\bigl(N_{X,\la}'\tp M,W^*\bigr)$ for any finite-dimensional module $W$. Using Corollary \ref{high-low}, one can show that $\Hom\bigl(N_{X,\la}'\tp M,W^*\bigr)\subset \Hom\bigl(N_{X,\la}\tp M,W^*\bigr)$
   as vector subspaces in $W^*$,  and the latter is isomorphic to $ \Hom_\k(X^*,W^*) \simeq\Hom_\k(W,X)$.
  Include $X$ in an irreducible decomposition $V=\op_i X_i$ over $\k$ and write
  $$\oplus_i\Hom\bigl((W,\Hom_\C(M,M_{X_i,\la}')\bigr)\simeq \Hom\bigl(W,\Hom_\C(M,V\tp M)\bigr)\simeq \Hom_\k(W,V).$$
 Since the rightmost term is isomorphic to $\op_i\Hom_\k(W,X_i)$, we get
 $\Hom\bigl((W,\Hom_\C(M_{X_i,\la}')\bigr)\simeq \Hom_\k(W,X_i)$ for all $X_i$ by dimensional reasons.
\end{proof}

 Denote by $\mathrm{Proj}(\Ac,\g)$ the additive category of equivariant finitely generated projective right $\Ac$-modules.
 The morphisms in $\mathrm{Proj}(\Ac,\g)$ commute with the left action of $U_q(\g)$ and the right action of $\Ac$.
 It is a module category over $\Fin_q(\g)$, since tensor product with a module from $\Fin_q(\g)$ preserves direct sums.
\begin{propn}
  The  $\Fin_q(\g)$-module categories $\mathrm{Proj}(\Ac,\g)$ and $\O(\Sbb_q^{2n})$ are equivalent.
\label{proj-pseudo-parabolic}
\end{propn}
\begin{proof}
The functor $\O(\Sbb_q^{2n})\to \mathrm{Proj}(\Ac,\g)$ is facilitated by Proposition \ref{Kostant}:
for a submodule $M_{X,\la}\subset V\tp M_\la$ and an invariant projector $\hat P\in V\tp M\to M_{X,\la}$, the $\Ac$-module $\hat P(V\tp \Ac)$ depends only on the image of $\hat P$. Clearly every object from $\mathrm{Proj}(\Ac,\g)$ is isomorphic to one obtained this way.
All $U_q(\g)$-module homomorphisms intertwine projectors and give rise to morphisms in $\mathrm{Proj}(\A,\g)$.
Conversely, the functor $\mathrm{Proj}(\Ac,\g) \to \O(\Sbb_q^{2n})$ is implemented by the
assignment $\Gamma=\hat P(V\tp \Ac)\mapsto \Gamma\tp_\Ac M = \hat P(V\tp M)$.
 It is straightforward to see that the functor $\O(\Sbb_q^{2n})\to \mathrm{Proj}(\Ac,\g)$
 is bijective on morphisms. It also preserves left tensor multiplication by modules from $\Fin_q(\g)$.
\end{proof}
\begin{remark}
\em
Let us comment on the star product on $\Sbb^{2n}$ constructed in \cite{M1}.
It is utilizing the isomorphisms
$\Hom(M,V\tp M)\simeq V^+_M[0]$ established in \cite{M1} by different methods.
The collection of these isomorphisms for all
irreducible $V$ describes the module structure of the trivial bundle $\Ac$ as the locally finite part of $\End(M)$.
This fact becomes  a special
case of Theorem \ref{thm-dir-sum} implying $\Hom(M_{X,\la},V\tp M)\simeq V^+_M[\xi]$, where $\xi$ is the highest weight of a $\k$-submodule
$X\subset V$ in the classical limit.
\end{remark}

In conclusion of this subsection, we compute the equivariant group $\Krm_0$  of $\Sbb^{2n}_q$.
Let $R(\g)$ denote the representation ring generated by isomorphism classes of finite-dimensional $\g$-modules.
It is known to be isomorphic to $\Z(\La)^W$, where $\La$ is the integral weight lattice,
$\Z(\La)$ the group ring, and $W$ the Weyl group \cite{FH}. The isomorphism is implemented by the  character $\Char$.
The ring $R(\k)$ has an additional structure of (left) $R(\g)$-module via the homomorphism
$R(\g)\to R(\k)$ induced by the restriction functor $\g\downarrow \k$.

The representation ring of finite-dimensional $U_q(\g)$-modules of quasiclassical type
is isomorphic to the classical representation ring $R(\g)$.
Let $\Krm_0(\Ac,\g)$ denote the Grothendieck group of the category $\Proj(\Ac,\g)$.
\begin{thm}
The group $\Krm_0(\Ac,\g)$ is a left $R(\g)$-module isomorphic to $R(\k)$.
\end{thm}
\begin{proof}
The isomorphism of Abelian groups readily follows from the equivalence of categories  $\Proj(\Ac,\g)\sim \O(\Sbb^{2n}_q)\sim \Fin(\k)$,
established by
Propositions \ref{quasi-parabolic category} and \ref{proj-pseudo-parabolic}.
The isomorphism of $R(\g)$-modules is a consequence of the isomorphisms $\Hom_\k(X,V)\simeq \Hom_{U_q(\g)}(M_{X,\la},V\tp M)$
for all $\k$-modules $X$ and all $U_q(\g)$-modules $V$, by Theorem \ref{thm-dir-sum}, 1).
\end{proof}
\noindent
In terms of formal characters, $\Krm_0(\Gamma_X)=\Char(X)\Char(M)$, by Theorem
\ref{thm-dir-sum}, 2).

\subsection{Coideal stabilizer subalgebra}
As for projective spaces in \cite{M4}, we give an alternative realization of quantum vector bundles  over spheres
in terms of quantum symmetric pairs.
Here we pass to the left $\Ac$-module version of vector bundles.

Let $\Tc$ denote the Hopf dual of $U_q(\g)$ that is a quantization of the function algebra on the spin group
covering $SO(2n+1)$.
It contains the quantum function algebra of the orthogonal group, which is generated by matrix coefficients
$T_{ij}$, $i,j=1,\ldots, 2n+1$, of the natural representation in $\C^{2n+1}$.
The matrix $T$  is invertible with $(T^{-1})_{ij}=\gm(T_{ij})$, where
$\gm$ stands for the antipode of $\Tc$.
There are two commuting left and right translation actions of $U_q(\g)$ on $\Tc$ expressed through the Hopf paring and the
comultiplication in $\Tc$ by
$$
h\tr a= a^{(1)}(h,a^{(2)}),\quad  a\tl h= (a^{(1)},h) a^{(2)},\quad  a\in \Tc, \quad h\in U_q(\g) .
$$
They are compatible with multiplication on $\Tc$.

Let $\Ru$ be a universal R-matrix of $U_q(\g)$.
Fix a representation of $U_q(\g)$ such that the image $R\in \End(\C^{2n+1})\tp \End(\C^{2n+1})$
of $\Ru$ is
proportional to the orthogonal R-matrix from \cite{FRT}.
The element $\Ru_{21}\Ru_{}$ commutes with the coproduct $\Delta(x)$ for all $x\in U_q(\g)$.
Its image $\Q$  in $\End(\C^{2n+1})\tp \End(M)$ is a matrix whose entries generate $\Ac$.
The matrix
$$
A=\quad
\left(\begin{array}{cccccc}
  q^{-2n}-q^{-1}& 0 & \ldots & 0 &q^{-n-\frac{1}{2}}c \\
  0 &-q^{-1}& 0 & 0 &0 \\
  \vdots & 0 &\ddots&0&\vdots\\
  0 & 0 & 0 &-q^{-1}&0  \\
  q^{-n-\frac{1}{2}}c^{-1}  & 0 & \ldots &0&0  \\
\end{array}
\right)\in \End(\C^{2n+1})
$$
with $c\in \C\backslash\{0\}$,
solves the Reflection Equation
$$
R_{21}A_1R_{12}A_2=A_2R_{21}A_1R_{12}\in \End(\C^{2n+1})\tp \End(\C^{2n+1}),
$$
where the indices mark the tensor factors. It also satisfies
other equations of the quantum sphere, cf. \cite{M2}, and defines a one-dimensional representation
$\chi\colon \Ac\to \C$, $\Q_{ij}\mapsto A_{ij}$.
The assignment $\Q\mapsto T^{-1}A T$ extends to an equivariant embedding  $\Ac\subset  \Tc$, where $\Tc$
is regarded as a $U_q(\g)$-module  under the left translation action. The character $\chi$ factors to
a composition of this embedding and the counit $\eps$.
The entries of the matrix $\Ru_{21}A_1\Ru_{12}\in \End(\C^{2n+1})\tp \U_q(\g)$ generate a left coideal subalgebra $\Bc\subset \U_q(\g)$,
such that $a\tl b=\epsilon(b)a$ for all $b\in \Bc$ and $a\in \Ac$. It  is a deformation of $U(\k')$,
where $\k'\simeq \s\o(2n)$.

Let $\hat P\in V\tp \Ac$ be an invariant idempotent.  The  projector  $P=\hat P_1\chi(\hat P_2)\in \End(V)$ commutes with $\Bc$ (see \cite{M4}).
\begin{propn}
\label{B-red}
\begin{enumerate}
  \item Every finite-dimensional right $U_q(\g)$-module $V$ is completely reducible over $\Bc$.
  \item Every irreducible $\Bc$-submodule in $V$ is a deformation of a classical $U(\k')$-submodule.
  \item Every $\Bc$-submodule in $V$ is the image of a $\Bc$-invariant projector $(\id\tp \chi)(\hat P)$, where
  $\hat P\in End(V)\tp \Ac$ is  a $U_q(\g)$-invariant idempotent.
\end{enumerate}
\end{propn}
The proof is similar to \cite{M4}. Note that Proposition \ref{B-red} holds for almost $q$
for each $V$.

Proposition \ref{B-red}  also  can be given a category theoretical interpretation.
Denote by $\Fin_q(\k')$ the additive category of $\Bc$-modules whose objects are submodules in some $V$ from $\Fin_q(\g)$
and morphisms are $\Bc$-invariant maps. Since it is semi-simple, it is Abelian and it is also a module category over $\Fin_q(\g)$.
Given an equivariant projective $\Ac$-module $\Gamma$, the assignment $\Gamma\mapsto \Gamma\tp_\Ac \C$ via the character $\chi$
is a functor $\mathrm{Proj}(\Ac,\g)\to \Fin_q(\k')$. This is an equivalence
of module categories.

\subsection{Vector bundles via symmetric pairs}
\label{SecSymPair}
The realization of quantum vector bundles as linear maps between pseudo-parabolic Verma modules has no classical analog.
In this section we follow a different approach presenting an associated vector bundle by $\Bc$-invariants
in the tensor product of $\Tc$ and a  $\Bc$-module.  This construction is quasi-classical: in the limit $q\to 1$
we recover the standard construction of equivariant vector bundles as induced modules.
It is convenient to pass to  left  $\Ac$-modules, which corresponds to the right coset picture.

It is known that every finite-dimensional right $U_q(\g)$-module $V$ is a left $\Tc$-comodule.
We use a Sweedler-like notation for the left coaction $V\to \Tc\tp V$,  $v\mapsto v^{(1)}\tp v^{[2]}$.
Then $v\tl h=(v^{(1)},h)v^{[2]}$ for $v\in V$, $h\in U_q(\g)$.

We define a left $U_q(\g)$-action on $V$ by $h\btr v=v\tl \gm(h)$, $h\in U_q(\g)$, $v\in V$,
and consider $\Ac\tp V$ as a left  $U_q(\g)$-module. The tensor product $\Tc\tp V$ is
also a left $U_q(\g)$-module with respect to the left translations on $\Tc$ and the trivial
action on $V$.

Let $\hat P\in \End(V)\tp \Ac$ be a $U_q(\g)$-invariant idempotent and $P=(\id \tp \chi)(\hat P)$ its $\Bc$-invariant image.
Denote by $X=PV$ the corresponding $\Bc$-submodule in $V$.
The subspaces $(\Ac\tp V)\hat P_{21}\subset \Tc\tp V$ (the diagonal $U_q(\g)$-action) and $(\Tc\tp X)^\Bc\subset \Tc\tp V$
(trivial $U_q(\g)$-action on $V$) are isomorphic as left $\Ac$-modules and
$U_q(\g)$-modules, \cite{M4}.
\begin{thm}
The $\Ac$-module $(\Tc\tp X)^\Bc$ is a quantization of the vector bundle $\Gamma(\Sbb^{2n},X)$.
\end{thm}
\begin{proof}
Similar to \cite{M4}.
\end{proof}
By considering vector (integer-spin) modules, one can restrict consideration to $\Tc$  the function algebra
on the quantum orthogonal group.

\appendix
\section{Parabolic Verma modules}
\label{Sec_ParVerMod}
It is known that a parabolic Verma module over a classical universal enveloping algebra $U(\g)$ of a simple Lie algebra $\g$
relative to a Levi subalgebra $\l\subset \g$ is locally finite over $U(\l)$, cf. \cite{Hum}. The goal of this auxiliary section
is to prove an analogous statement for quantum groups. We are not aware if this topic is covered in the literature
and include it for completeness.

We start with a few general definitions.
Let $A$ be an associative algebra with unit.
Fix an element $a\in A$ and define its quasi-normalizer $\mathrm{QNorm}(a)\subset A$ as the subset of elements
$b\in A$ such that $a^{m+k} b=b_k a^k$ for some $m\in Z_+$, $b_k\in A$ and all positive $k\in \Z_+$.
Clearly it is a unital subalgebra in $A$.

Here is a way of checking if $b\in \mathrm{QNorm}(a)$.
Consider a sequence of complex numbers $(q_i)_{i=1}^\infty$ and define by induction
$D_a^0(b)=b$, $D_a^n(b)=[a,D_a^{n-1}(b)]_{q_n}$ for $n>0$.
\begin{propn}
  Suppose that $D_a^{m+1}(b)=0$ for some $m\geqslant 0$. Then $b\in \mathrm{QNorm}(a)$.
\label{Leibnitz}
\end{propn}
\begin{proof}
  One can prove by induction that $a^nb=\sum_{k=0}^{n}c_{n,k}D^k_a(b) a^{n-k}$ for all positive integer $n$, where $c_{n,k}$ are some complex coefficients.
It follows that $D_a^l(b)=0$ for $l>m$, and
$a^nb=(\sum_{k=0}^{m}c_{n,k}D^k_a(b) a^{m-k})a^{n-m}$, as required.
\end{proof}
We call an element $a\in A$ quasi-normal if $\mathrm{QNorm}(a)= A$. Their set is stable under the group of automorphism
of $A$. It is straightforward that  invertible elements of $A$ are quasi-normal, as well as elements of the center of $A$.
We aim to prove that Lusztig root vectors in $U_q(\g)$ are quasi-normal.

Recall that a total order in  $\Rm^+$ is called normal if every sum of positive roots is between the summands.
Fix such an order and set $\al^i\in \Rm^+$, $i=1,\ldots, N=\# \Rm^+$.
Then there exists a system of "root vectors" $\{e_i\}_{i=1}^N$ of weights $\al^i$
such that the monomials $\{e_1^{m_1}\ldots e_N^{m_N}\}_{m_i\in \Z_+}$ deliver a PBW basis in $U_q(\g_+)$. Similarly there are
negative root vectors $f_i\in U_q(\g_-)$ giving rise to a PBW basis of ordered monomials in $U_q(\g_-)$.

\begin{propn}
  The Lusztig root vectors are quasi-normal.
  \label{Lusztig-q-norm}
\end{propn}
\begin{proof}
  It is sufficient to prove that generators of $U_q(\g)$ belong to $\mathrm{QNorm}(e_i)$ for each
  $e_i$. First suppose that $e_i=e_\al$ for a simple root $\al$. Then clearly $U_q(\h)\subset \mathrm{QNorm}(e_\al)$.
  Furthermore, $e^n_\al f_\bt=(\dt_{\al\bt} f_\bt e_\al + [n]_q[h_\al+n-1]_q) e_\al^{n-1}$, so $U_q(\g_-)\subset \mathrm{QNorm}(e_\al)$.
  Finally, the Serre relations can be written as $D_{e_\al}^{m}(e_\bt)=0$ for certain integer $m$ and a sequence $(q_i)_{i=1}^m$ depending on
  $\al$ and $\bt$. Then by Proposition \ref{Leibnitz},
  $U_q(\g_+)\subset \mathrm{QNorm}(e_\al)$, and therefore $e_\al$ is quasi-normal. Now observe that
  compound root vectors are obtained from the Chevalley generators via automorphisms of $U_q(\g)$, \cite{ChP}. This completes the proof.
\end{proof}
Let $\l$ be a Levi subalgebra in $\g$ with a basis of simple roots $\Pi^+_\l\subset \Pi^+_\g$.
One can choose a normal order so that $\Rm^+\backslash \Rm^+_\l<\Rm^+_\l$.
Then  $\Span\{f_1^{m_1}\ldots f_l^{m_l}\}_{m_i\in \Z_+}$, where $l=\# (\Rm^+\backslash \Rm^+_\l)$,
is a subalgebra, cf. \cite{KT1}, Proposition 3.3 (see also \cite{Ke}). It is a quantization
of the $U(\n_-)$, where $\n_\pm$ are the nil-radicals of the parabolic Lie algebras $\p_\pm=\l+\g_\pm$.
Similarly $\Span\{e_l^{m_l}\ldots e_1^{m_1}\}_{m_i\in \Z_+}$ is
a subalgebra,   $U_q(\n_+)$.
The PBW factorization then yields  $U_q(\g)=U_q(\n_-)U_q(\p_+)$ with $U_q(\p_+)=U_q(\l)U_q(\n_+)$.

Now we use an alternative construction of parabolic Verma modules  via parabolic induction.
Given $\check X\in \Fin_q(\l)$ (relax the quasi-classical weight convention) make it a $U_q(\p_+)$-module by setting it trivial on $U_q(\n_+)$.
Then $M_{\check X}=U_q(\g)\tp_{U(\p_+)}\check X$ is said to be parabolically induced from $\check X$.
In the special case of irreducible  $\check X$, this construction recovers
the parabolic Verma module as the quotient of the ordinary Verma module whose  highest weight is that of $\check X$.
The PBW factorization of $U_q(\g_-)$ implies that  $M_{\check X}$ is a free $U_q(\n_-)$-module generated by $\check X$.

Recall that an $A$-module $V$ is called locally finite if for each $v\in V$ the cyclic submodule $Av\subset V$ is finite-dimensional. An element $a\in A$ is called locally nilpotent if $a^nv=0$ for all $v\in V$ and some positive $n\in \Z_+$ depending on $v$.
\begin{propn}
  The parabolic module $ M_{\check X}$ is locally finite over $U_q(\l)$.
  \label{loc_fin_parabolic}
\end{propn}

\begin{proof}
It is sufficient to prove that $ M_{\check X}$ is locally finite over $U_q(\l_-)$ since it is clearly so over  $U_q(\h)$ and $U_q(\l_+)$.
Indeed, every $f_\al$ with $\al\in \Rm^+_\l$ is locally nilpotent on $\check X$ and therefore on $ M_{\check X}\simeq U_q(\n_-) \check X$ since
it is quasi-normal. Therefore all but a finite number of PBW-monomials in $U_q(\l_-)$ vanish on every $v\in M_{\check X}$.
This implies the statement.
\end{proof}
It follows that $ M_{\check X}$ is completely reducible over $U_q(\l)$ and is an infinite direct sum of irreducible finite-dimensional
modules. As another consequence, $U_q(\n_-)$ is a locally finite adjoint $U_q(\l)$-module because it is isomorphic to the
parabolic Verma module induced from the trivial representation of $U_q(\p_+)$.
\section{Tensor product of modules of  highest and lowest weight}
For reader's convenience, we prove some facts that we use in the proof of Proposition
\ref{homomorphism}.
Recall that a lower Verma module is freely generated over $U_q(\g_-)$ by its lowest weight vector.
\begin{lemma}
  Let $\hat M_\mu$ be a  Verma module of highest weight $\mu$ and $\hat N_\nu$ a Verma module of lowest weight $\mu$.
  Then the tensor product $\hat M_\mu \tp \hat N_\nu$ is isomorphic to the $U_q(\g)$-module $\Ind_\h^\g(\C_{\mu+\nu})$
  induced from the $U_q(\h)$-module $\C_{\mu+\nu}$.
\label{upper-lower}
\end{lemma}
\begin{proof}
  Let us prove that the homomorphism $\Ind_\h^\g(\C_{\mu+\nu})\to \hat M_\mu \tp \hat N_\nu$ determined by the assignment $d\colon x\mapsto \Delta(x) (1_{\mu}\tp 1_\nu)$, $x\in U_q(\g)$, is an isomorphism. Introduce a $\Z$-grading on $U_q(\g_\pm)$ by assigning degree $1$ to the
  Chevalley generators and denote by $U^k_\pm$ their homogeneous component of degree $k$. This makes the vector spaces $\Ind_\h^\g(\C_{\mu+\nu})=U_q(\g_-)U_q(\g_+)1_{\mu+\nu}$ and
  $\hat M_\mu \tp \hat N_\nu$ graded by the semi-group $\Z^2_+$. The map $d$ is an isomorphism on the components $(0,j)$ for all $j\in \Z_+$.
  Since $f_\al$, $\al\in \Pi^+_\al$, are quasi-primitive, the map $d$  sends the $U^i_-U^j_+1_{\mu+\nu}$ onto the
  $U^i_-1_{\mu}\tp U^j_+1_{\nu}$ modulo $\sum_{l<i,k\leqslant j}U^l_-1_{\mu}\tp U^k_+1_{\nu}$.
  Induction on $i$ proves that $d$ is surjective and therefore injective on $\sum_{l\leqslant i,k\leqslant j}U^l_-U^k_+1_{\mu+\nu}$ for all $j$, as the graded components are finite-dimensional.
\end{proof}
It follows from Lemma \ref{upper-lower} that $\Hom (\hat M\tp \hat N,V)\simeq V[\mu+\nu]$ for any $V$.
Suppose $\{f_i\}\subset  U_q(\g_-)$ are such that $f_i1_\mu$ are singular vectors in $\hat M_\mu$ and let $M$ be the quotient of $\hat M$ by
the sum of submodules they generate. Similarly let $\{e_j\} \subset  U_q(\g_+)$ be such that $e_i1_\nu $ are killed by all Chevalley generators $f_\al$
and let $N$ denote the corresponding quotient of $\hat N_\nu$.
\begin{corollary}
For any  finite-dimensional module $V$, $\Hom (M\tp N,V)$ is isomorphic to $\Bigl(\cap_{i,j}\ker_V(f_i)\cap \ker_V(e_j)\Bigr)[\mu+\nu]$.
\label{high-low}
\end{corollary}
\begin{proof}
  The module $M\tp N$ is a quotient of $\hat M_\mu\tp \hat N_\nu$ by the sum of $U_q(\g_-)f_i1_\mu\tp \hat N_\nu$
  and $\hat M_\mu\tp U_q(\g_+)e_j1_\nu$, which disappear in a homomorphism to $V$. By Lemma \ref{upper-lower}, all these submodules are
  induced from one-dimenstional $U_q(\h)$-modules.
   Therefore $\Hom (M\tp N,V)$ is in bijection with  vectors in $V[\mu+\nu]$ that are killed by all $f_i$ and all $e_j$.
\end{proof}

\end{document}